\documentclass[14pt]{amsart}
\usepackage{lscape} 
\usepackage{enumerate}
\usepackage{xcolor}
\usepackage{amssymb,amsmath,mathtools}
\usepackage{lscape, setspace}
\usepackage{enumerate}
\usepackage{multirow}
\usepackage{booktabs}
\usepackage{latexsym}
\usepackage{amscd,amsfonts}
\usepackage{graphicx}
\usepackage{color}
\usepackage{pdflscape}
\usepackage{longtable}
\usepackage{epsfig}
\usepackage{tikz}
\usepackage{adjustbox}
\usepackage{svg}
\usepackage{float}

\theoremstyle{plain}\swapnumbers
\newtheorem{theorem}{Theorem}

\newtheorem{corollary}[theorem]{Corollary}
\theoremstyle{definition}

\newtheorem{lemma}[theorem]{Lemma}

\theoremstyle{remark}
\newtheorem{remark}[theorem]{Remark}

\numberwithin{equation}{section}
 \DeclareMathOperator{\Aut}{Aut}

 \DeclareMathOperator{\GL}{GL}
\DeclareMathOperator{\id}{id}

\DeclareMathOperator{\Mat}{Mat}

\newcommand{\F}{\mathbb{F}}
\newcommand{\Z}{\mathbb{Z}}
\newcommand{\N}{\mathbb{N}}

\makeindex
\begin{document}

\title[Irreducible components of the varieties of Jordan superalgebras]
{Irreducible components of the varieties of Jordan superalgebras of types $(1,3)$ and $(3,1)$.}

\author[Hern\'andez]{Ma. Isabel Hern\'andez}
\address{SECIHTI-CIMAT Unidad M\'erida}
\email{isabel@cimat.mx }

\author[Martin]{Mar\'ia Eugenia Martin }
\address{Universidade Federal do Parana}
\email{eugenia@ufpr.br}

\author[Rodrigues]{Rodrigo Lucas Rodrigues}
\address{Universidade Federal do Cear\'a}
\email{rodrigo@mat.ufc.br }

\thanks{The first author was
supported by Grant CONAHCYT A1-S-45886.}
\begin{abstract}
We describe the variety of Jordan superalgebras of dimension $4$ whose even part is a Jordan algebra of dimension $1$ or $3$. We prove that the variety is the union of Zariski closures of the orbits of $11$ and $21$ rigid superalgebras, respectively. In both cases, the irreducible components of the varieties are described. Furthermore, we exhibit a four-dimensional solvable rigid Jordan superalgebra, showing that an analogue to the Vergne conjecture for Jordan superalgebras does not hold.

%Furthermore, M. Vergne conjectured that a nilpotent Lie algebra cannot be rigid in the variety of Lie algebras. Analogous assertions can be postulated for other varieties. In particular, no examples of solvable rigid Jordan algebras were found until now. In this paper, we show that the conjecture does not hold in the superalgebra case; namely, we exhibit a four-dimensional solvable rigid Jordan superalgebra.

\end{abstract}
\maketitle
\section{Introduction}

%%%%%%%%%
%%
%% Introduction of the  paper:
%% Variety of Jordan superalgebras of dimension four (possible title)
%% 
%%%%%%%%%

The algebraic  classification of different types of low-dimensional algebras represents
a classical problem of significant interest in algebraic studies. In recent years, there has been a growing 
 significance in their {\it geometric} classification. 

Several noteworthy results have been achieved regarding the degenerations between low-dimensional algebras in a variety defined by polynomial identities. A crucial focus in this line of inquiry is to find the rigid algebras of the variety. These algebras hold significant interest since the closures of their orbits, under the action of the general linear group, constitute irreducible components within the considered variety, with respect to  the Zariski topology.

Numerous works have delved into the challenging task of identifying the irreducible components of low-dimensional algebraic varieties. Namely, the works of P. Gabriel \cite{finiterepresentatiotypeisopen} and G. Mazzola \cite{mazzola} have describe all rigid algebras for the variety of unitary associative algebras of  
 dimensions four and five, respectively. Moreover, the geometric classiﬁcation of Lie algebras, is known for dimension up to six; see, for example, %\cite{kirillov} and \cite{Burde1999}.
\cite{Burde1999} and \cite{kirillov}.

As for the variety of Jordan algebras,  references began in $2005$ with 
%the paper 
\cite{irynashesta}, by I. Kashuba and I. Shestakov, related to three-dimensional Jordan algebras. It was followed by %the works  
\cite{kashubamartin} and \cite{MartinJor3}, by I. Kashuba and M. E. Martin, who studied the geometric classification of four-dimensional 
Jordan algebras over an algebraically closed field and three-dimensional real 
 Jordan algebras,  respectively. The authors showed that any irreducible component in these varieties is determined by a rigid algebra. 

In \cite{ancocheabermudes}, %the authors 
 it was established that the subvariety of nilpotent Jordan  algebras is irreducible for dimensions up to three, and  it admits two algebraic components in dimension four. In $2018$, in \cite{MartinJorN5}, was determined that the subvariety of nilpotent Jordan algebras of dimension five is the union of ﬁve irreducible components, with four rigid algebras, and the closure of the union  of a certain inﬁnite family gives rise to the ﬁfth component.

Others non-associative algebras were recently considered. Namely,  in \cite{GeoPowerAssoc} (2020), the variety of four-dimensional commutative power associative algebras was studied by R. L. Rodrigues et al.,
%A. Papa Neto, R. Rodrigues, and E. Quintero Vanegas, 
 where it was proved that this variety decomposes into twelve irreducible components that correspond to the Zariski closure of the orbit of rigid algebras, two of  them are non-Jordan algebras.  In \cite{DeformPoissonAlgebras}  (2023), the authors  proved that the variety of four-dimensional nilpotent Poisson algebras has five irreducible components corresponding to three rigid Poisson algebras and two infinite families. 
See also analogous works about pre-Lie algebras in \cite{preLie2009} (2009), Filippov algebras in  \cite{Filippov2011}
 (2011), Leibniz algebras in \cite{Casas2013} (2013), Novikov algebras in \cite{Novikov2014} (2014), among others.

However, the literature on this topic about varieties of superalgebras or graded algebras is scarce. 
The theory adapted to the case of {\it graded} Jordan algebras and Jordan {\it superalgebras}  was first considered in \cite{2003Kashuba_Graded} (2003).

In $2014$, A. Armour and Y. Zhang in \cite{Armour2014} determined that the variety of unitary associative superalgebras of dimension four and of type $(3,1)$ decomposes into six irreducible components, of type $(2,2)$ decomposes into between  eight and ten irreducible components and the variety of these superalgebras of type $(1,3)$  is indecomposable.
 
Degenerations of four-dimensional Lie superalgebras  
of type $(2,2)$ were 
 studied in \cite{DegLieSuperalg}, by  M. A. Alvarez and I. Hernández in $2020$. The authors 
 obtained the Zariski closure of every orbit and proved that this variety is the union of seven irreducible components, three of  them
are the Zariski closures of rigid Lie superalgebras.

The only reference  we know about the geometric classification of Jordan superalgebras is given in dimension three, in  \cite{Geom_SuperJAdim3}, a work of M. A. Alvarez et al.,
 published in $2019$,  where it was
 described the seven irreducible componentes of the variety of Jordan superalgebras of type $(1,2)$ and the same number of irreducible 
 components of the variety of Jordan superalgebras of type $(2,1)$. All irreducible components correspond to the Zariski 
 closure of the orbit of some rigid superalgebra.

The goal of this paper is to explore the geometry of Jordan superalgebras of dimension four belonging to  the
types $(1,3)$ and $(3,1)$, over an algebraically closed ﬁeld of characteristic $0$. Our specific focus involves identifying deformations between two superalgebras, determining the rigid superalgebras, and describing the irreducible components of these varieties.

%{For the first type, we demonstrate that the variety is a union of Zariski closures of the orbits of $11$ rigid superalgebras, with one of them being solvable. The irreducible components of this variety have been thoroughly detailed. {\bf It is noteworthy that, in the classical case, 
%there are no known  no examples of nilpotent/solvable rigid Jordan algebras are known so far}. For the second type, we establish that the variety is the union of Zariski closures of the orbits of $21$ rigid superalgebras.} 

%%% Outra opção para o parágrafo anterior- Eugenia 20/2:

For the first type, we demonstrate that the variety is a union of Zariski closures of the orbits of $11$ rigid superalgebras. Furthermore, in \cite{VergneConj} (1970), M. Vergne conjectured that there is no complex $n$-dimensional nilpotent Lie algebra which is rigid in the variety of all $n$-dimensional complex Lie algebras. Similar conjectures could  potentially be  applied to other varieties. In particular, no examples of nilpotent or solvable rigid Jordan algebras are known so far. In this paper, we show that the conjecture does not hold in the superalgebra case; namely, we exhibit a four-dimensional solvable rigid Jordan superalgebra of type $(1,3)$. For the type $(3,1)$, we establish that the variety is the union of Zariski closures of the orbits of $21$ rigid superalgebras.

As a direct consequence we achieve the geometric classification of associative and supercommutative superalgebras of dimension four and types $(1,3)$ and $(3,1)$, resulting in five and two irreducible components, respectively. Additionally, we determine that the subvariety of nilpotent Jordan superalgebras of dimension four and type $(1,3)$  has two irreducible components, while for type $(3,1)$ it is indecomposable. 
We emphasize that in a forthcoming paper, we will complete the four-dimensional classification, presenting the orbits of the Jordan superalgebras of type $(2,2)$.

The paper is organized as follows. Section 2 begins with some preliminaries on Jordan superalgebras, we deﬁne the variety of Jordan superalgebras, present useful invariants and discuss the (non-)existence of deformations. In Section 3, we  describe deformations between the Jordan  superalgebras of type $(1,3)$ and characterize the irreducible components of such variety. Finally, in Section 4,  we proceed similarly to Jordan superalgebras of type $(3,1)$.

\section{Preliminaries}

Throughout this paper, $\F$ denotes an algebraically closed field of characteristic $0$. A  {\it Jordan algebra} over $\F$ is an algebra $\mathcal{J}$  satisfying the identities 
$xy=yx$ (commutativity) and $(x^{2}y)x = x^2(yx)$ (the Jordan identity). 
 By linearization techniques, we  
obtain
the equivalent multilinear identity
\
\begin{eqnarray*}
(wx)(yz)+ (wy)(xz)+ (wz)(xy)
- x (w(yz ))-y(w(xz)) - z(w(xy))=0.
 \end{eqnarray*}

A {\it superalgebra} $A$ is  
a $\mathbb{Z}_2$-graded algebra. 
This means that $A$  
can be decomposed into  
the direct sum of subspaces $A=A_0 \oplus A_1$ such that $A_iA_j\subseteq A_{i+j}$, for $i, j \in \mathbb{Z}_2$.  In this way,  
$A_0$  is a subalgebra and $A_1$ is an $A_{0}$-bimodule. The elements of $A_0\setminus\{ 0 \}$ (respectively,  of $ A_1 \setminus \{0\}$) are 
said to be even (respectively, odd) and 
  $x_i$
in $A_i\setminus\{ 0 \}$, $i \in \mathbb{Z}_2$, is called 
homogeneous and of degree $i$. We denote it by   
$|x_i|=i$. The pair $(\dim_\mathbb{F} (A_0), \dim_\mathbb{F} (A_1))$ is  
 called the type of $A$. 
 
A morphism of superalgebras $A$ and $A^\prime$ is a linear map $\Phi\colon A \to A^\prime$ such that $\Phi$ is even (i.e., $\Phi(A _i) \subset A^\prime_i $,  for $ i \in \mathbb{Z}_2$) and $\Phi (xy)=  \Phi(x) \Phi(y)$, for all $x,y \in A$.

Let $\mathcal{G}= {\rm alg}\{1, e_i | 1\leq i, \;  e_ie_j=-e_je_i \}$  be  the Grassmann algebra.  It can be decomposed as
$\mathcal{G}=\mathcal{G}_0 \oplus \mathcal{G}_1$,  and it is an  associative  $\mathbb{Z}_2$-graded algebra, where $\mathcal{G}_0$ (respectively, $\mathcal{G}_1$) are subspaces spanned by the products of even (respectively, odd) length. In addition, $\mathcal{G}$  satisfies 
$uv=(-1)^{|u||v|}vu$,  the supercommutativity, for  all homogeneous elements $u, v \in \mathcal{G}$. 
 For a superalgebra $A = A_{0} \oplus A_{1}$, we define its {\it Grassmann envelope} by
$
\mathcal{G}(A)= (\mathcal{G}_0 \otimes A_0)  \oplus  (\mathcal{G}_1 \otimes A_1),
$  
where the multiplication is 
 given by $(x\otimes u)(y\otimes v):= xy\otimes uv$, for all $x\otimes u \in \mathcal{G}_i \otimes A_i$,  $y\otimes v \in \mathcal{G}_j \otimes A_j$, and $i,j \in \mathbb{Z}_2$.

Let $\mathcal{ V}$ be a variety of algebras over $\F$ defined by a set of (multilinear) identities. A superalgebra $A=A_0\oplus A_1$ is 
 said to be a $\mathcal{ V}$-superalgebra if its  Grassmann envelope $G(A)$ lies in $\mathcal{ V}$.  In particular,  if  $A =A_0\oplus A_1$ is a $\mathcal{ V}$-superalgebra, then the algebra $A_0$ lies in $\mathcal{ V}$ and $A_1$ is  an $A_0$-bimodule in the class $\mathcal{ V}$.  

Let $\{ P_\alpha\}_{\alpha\in I}$ be the set of multilinear identities that defines the variety $\mathcal{ V}$. In general, the corresponding set of superidentities $\{ P^s_ \alpha\}_{ \alpha \in I}$ which defines the variety of $\mathcal{ V}$-superalgebras can be obtained taking into account that if two homogeneous adjacent elements $x,y$ are exchanged, then the corresponding term is multiplied by $(-1)^{|x||y| }$. From this, we can deduce the corresponding superidentities defining the variety of Jordan superalgebras.

A {\it Jordan superalgebra}  is a superalgebra  $\mathcal{J}=\mathcal{J}_0 \oplus \mathcal{J}_1$ satisfying the following superidentities:
\begin{enumerate}[(i)]
\item  Supercommutativity:  $xy=(-1)^{|x||y|} yx;$ 
\item Jordan superidentity: 
\begin{eqnarray*}
(wx)(yz)+(-1)^{|x||y|}(wy)(xz)+(-1)^{(|x|+ |y|)|z|  }(wz)(xy)& \\
-(-1)^{ |w| |x|} x (w(yz ))- (-1)^{| y| (|w| +|x| )}  y(w(xz)) -  (-1)^{ |z|(|w|+|x|+|y|)} z(w(xy))=0,
\end{eqnarray*}
\end{enumerate} 
for  all $x,y,z,w \in (\mathcal{J}_0\cup \mathcal{J}_1)\setminus \{0\}$.
%\end{definitihave on}

%\dsubsection{ The variety of Jordan superalgebras of type $(m,n)$.}

Let $V=V_0 \oplus V_1$ be a $\Z_2$-graded vector space with a fixed  homogeneous basis  
$\left\{e_1, \ldots, e_m, f_1, \ldots, f_n\right\}$. In order to endow $V$ with a Jordan superalgebra structure, we can specify a set of structure constants $(\alpha_{ij}^k, \beta_{ij}^k, \gamma_{ij}^k) \in    \F^{m^3+ 2mn^2}$, where  
$$e_i e_j = \sum_{k=1}^{m} \alpha_{ij}^k e_k, \quad 
e_i f_j = \sum_{k=1}^{n} \beta_{ij}^k f_k, \quad 
f_i f_j = \sum_{k=1}^{m} \gamma_{ij}^k e_k.$$
Since every set of structure constants must satisfy the polynomial identities given by the supercommutativity and Jordan superidentity, it follows that the set of all Jordan superalgebras of type $(m,n)$ defines an affine variety in  $\F^{m^3+ 2mn^2}$, denoted by $\mathcal{JS}^{(m, n)}$. Note that a point $(\alpha_{ij}^k, \beta_{ij}^k, \gamma_{ij}^k)\in \mathcal{JS}^{(m, n)} $ represents, in the fixed basis, a Jordan superalgebra $\mathcal{J}$ over $\F$ of type  $(m,n)$. 

On the other hand, since morphisms of $\Z_2$-graded algebras are even maps,
there is a natural ``change of basis" action of 
the group 
$G= \GL_m(\F) \times \GL_n(\F)$ on $\mathcal{JS}^{(m, n)}$ yields a one-one correspondence between $G$-orbits $\mathcal{J}^G$ on $\mathcal{JS}^{(m, n)}$ and  the isomorphism classes of Jordan superalgebras of type $(m,n)$.   

%\begin{definition}
%Let $\mathcal{J}, \mathcal{F} \in \mathcal{JS}^{(m, n)}$. The superalgebra $\mathcal{J}$ is called a \textit{deformation} of 
%$\mathcal{F}$, and is denoted by $\mathcal{J} \to \mathcal{F}$, if $\mathcal{F}^G \subset \overline{\mathcal{J}^G} $,   
%where  $\overline{\mathcal{J}^G}$ denotes the Zariski closure of the $G$-orbit of $\mathcal{J}$. 
%\end{definition}

%\begin{definition}
Let $\mathcal{J}, \mathcal{J}^\prime \in \mathcal{JS}^{(m, n)}$. The superalgebra $\mathcal{J}$ is called a \textit{deformation} of 
$\mathcal{J}^\prime$, and  it is denoted by $\mathcal{J} \to \mathcal{J}^\prime$, if $\mathcal{J}^{\prime G} \subseteq \overline{\mathcal{J}^G} $,   
where  $\overline{\mathcal{J}^G}$ denotes the Zariski closure of the $G$-orbit of $\mathcal{J}$.

The concept of deformation is closely related to that degeneration,  in the sense that $\mathcal{J}$ degenerates to $\mathcal{J}^\prime$ corresponds to $\mathcal{J}^\prime$ deforms into $\mathcal{J}$. In this paper, we adopt the terminology of deformations but in some literature cited in this work, the results are given in terms of degenerations.  
 
 We call $\mathcal{J} \in \mathcal{JS}^{(m, n)}$ 
 {\it (geometrically) rigid} if the orbit $\mathcal{J}^G$ is Zariski-open set in $\mathcal{JS}^{(m, n)}$.  Any deformation $\mathcal{J}^\prime$ of a rigid superalgebra $\mathcal{J}$ satisfies $\mathcal{J}^\prime  \simeq \mathcal{J}$.

In the same way as any affine variety, $\mathcal{JS}^{(m, n)}$ could be decomposed into irreducible components. Notice that the  Zariski-closure of an open orbit is an irreducible component of the variety $\mathcal{JS}^{(m, n)}$,  
 (see \cite{tesejenny}, Proposition 4.36, p.69).

We will describe the deformations of Jordan superalgebras using  
``one-parameter family of deformations'', a concept introduced by  M. Gerstenhaber, in \cite{gerstenhaber}. Namely, let $\mathcal{J}$ be an $(m,n)$-dimensional Jordan superalgebra and consider $g(t)\in \Mat_m(\F(t)) \times \Mat_n(\F(t))$, where $\F(t)$ denote the Laurent polynomials in the variable $t$. Suppose  that for any non-zero $t\in \F$   we have $g(t )\in G$ and let $\mathcal{J}_t=(\alpha_{ij}^k(t), \beta_{ij}^k(t), \gamma_{ij}^k(t))$  be the Jordan superalgebra obtained by applying the change of basis $g(t)$ to $\mathcal{J}$. Then, $\mathcal{J}$ is a deformation of  the Jordan superalgebra $\mathcal{J}^\prime=(\alpha_{ij}^k(0), \beta_{ij}^k(0), \gamma_{ij}^k(0))$. In particular, a superalgebra $\mathcal{J}$ is rigid if any $g (t )$ satisfying the above conditions defines the algebra $\mathcal{J}_t$ isomorphic to $\mathcal{J}$ for every $t \in \F$. Notice that any superalgebra is a deformation of the trivial one taking $g(t)= t^{-1} \id$.

The relation $ \mathcal{J} \to  \mathcal{J}^\prime$  provides a partial order in $\mathcal{JS}^{(m,n)}$, and for a finite number of orbits, the Hasse diagram yields all of the information about the studied variety.

The absence of a deformation for a given pair of superalgebras will result from the failure to meet one of the conditions in Lemma \ref{Lema:nondeformation} below,  where for a Jordan superalgebra $\mathcal{J}$ with structure constants $(\alpha_{ij}^k, \beta_{ij}^k, \gamma_{ij}^k)$,  we denote by $a(\mathcal{J})$ and $F(\mathcal{J})$ the Jordan superalgebra with structure constants $(0, 0, \gamma_{ij}^k)$ and $(\alpha_{ij}^k, \beta_{ij}^k,0)$, respectively. The powers $\mathcal{J}^r$ are defined recursively by $\mathcal{J}^1 = \mathcal{J}$ and   $\mathcal{J}^r = \mathcal{J}^{r-1}\mathcal{J} + \mathcal{J}^{r-2}\mathcal{J}^2 + \cdots + \mathcal{J} \mathcal{J}^{r-1}$, for all $r\in \N.$ In every case $\mathcal{J}^r = (\mathcal{J}^r)_0 \oplus (\mathcal{J}^r)_1$.

%\begin{itemize}
%\item[$\bullet$] $a(\mathcal{J})$ denotes the Jordan superalgebra with the same underlying vector superspace as $\mathcal{J}$ and multiplication table given by $$f_i f_j = \sum_{k=1}^{n} \gamma_{ij}^k e_k,$$ with all other products being $0$.
%\item The powers $\mathcal{J}^r$ are defined recursively by $\mathcal{J}^1 = \mathcal{J}$ and  $$ \mathcal{J}^r = \mathcal{J}^{r-1}\mathcal{J} + \mathcal{J}^{r-2}\mathcal{J}^2 + \cdots + \mathcal{J} \mathcal{J}^{r-1}.$$ In every case $\mathcal{J}^r = (\mathcal{J}^r)_0 \oplus (\mathcal{J}^r)_1$.
%\item The Jordan superalgebra $\mathcal{F}(J)= $  
%\item[$\bullet$] The {\it Burde invariant} $$c_{i,j} = \dfrac{\mathrm{tr}(L(x)^i) \cdot \mathrm{tr}(L(y)^j)}{\mathrm{tr} (L(x)^i \cdot L(y)^j)},$$ where $L(x)$ is the left multiplication by $x$ and $\mathrm{tr}$ denotes the trace, is defined as a quotient of two polynomials in the structure constants of $\mathcal{J}$, for all $x, y \in \mathcal{J}$ such that both polynomials are not zero and $c_{i,j}$ is independent of the choice of $x, y$. 
%\end{itemize}

The following invariants have already been used in the study of the varieties of Lie superalgebras and  Jordan superalgebras,  see \cite{DegLieSuperalg} and \cite{Geom_SuperJAdim3}. 

\begin{lemma} \label{Lema:nondeformation}
Let $ \mathcal{J},  \mathcal{J^\prime} \in \mathcal{JS}^{(m,n)}$. If 
$ \mathcal{J} \to  \mathcal{J}^\prime$, then the following conditions hold:  
\begin{enumerate}[(i)]
\item $\dim(\Aut( \mathcal{J})) < \dim(\Aut( \mathcal{J}^\prime))$. 
\item $ \dim( \mathcal{J}^r)_i \geq \dim ({ \mathcal{J}^\prime}^{r})_i$, for $i\in \mathbb{Z}_2$.
\item $ \mathcal{J}_{0}  \to \mathcal{J}_0^\prime$.
\item $a( \mathcal{J}) \to a( \mathcal{J}^\prime)$.
\item $F( \mathcal{J}) \to F( \mathcal{J}^\prime)$.
%\item If the Burde invariant exists for  $\mathcal{J}$ and $ \mathcal{F}$, then it is the same for $ \mathcal{J}$ and $ \mathcal{F}$.
\item If $ \mathcal{J}$ is associative, then $ \mathcal{J}^\prime$ is also associative. Moreover, if $ \mathcal{J}$ satisfies a polynomial identity, then $\mathcal{J}^\prime$ satisfies the same polynomial identity.
\end{enumerate} 
\label{lemma:invariants}
\end{lemma}

\begin{remark} \label{Remark:como algebras} Let $\mathcal{J}, \mathcal{J}^\prime \in\mathcal{JS}^{(m,n)}$ such that $\mathcal{J}$ and $\mathcal{J}^\prime$ have structure of Jordan algebras. If 
 $\mathcal{J} \not \to \mathcal{J}^\prime$ as algebras, then 
 $\mathcal{J} \not \to \mathcal{J}^\prime$ as superalgebras. In our investigation, we use \cite{tesejenny} to find out when one Jordan algebra does not deform into another.
\end{remark}

\section{Jordan Superalgebras of Type $(1,3)$}

%%%
%%%  Jordan Superalgebras of ty[e (1|3)
%%%%

In this section, we  investigate 
the variety $\mathcal{JS}^{(1,3)}$.   I. Hernández et al. in \cite{SuperPowerAssoc4} have provided a concrete list of non-isomorphic commutative power-associative (Jordan, in particular) superalgebras up to dimension $4$ over an algebraically closed field  of characteristic prime to $30$. There exist, up to isomorphism, $20$  Jordan superalgebras of type $(1,3)$ (see \cite{SuperPowerAssoc4}, Table 15, p. 1630). 
Table \ref{table:JSA_1_3} gives representatives for  isomorphism classes and some additional useful information, 
 namely, the dimension of the automorphism group of each superalgebra, and we 
indicate by ``A'' if the superalgebra is associative and ``NA'' otherwise. 
 The 
 superscript “N” in $(1,3)_i^N$ indicates that the superalgebra is nilpotent. 

 \begin{longtable}[H]{rlcc}
 \caption{\label{table:JSA_1_3}Jordan superalgebras of type (1,3)}\\
Label  & Multiplication table & $\dim(\Aut(\mathcal{J})) $ & \\
\hline
\endhead
$(1,3)_1$:&$\;\;e_1^2 = e_1$. & $9$ & A \\
$(1,3)_2$:&$\;\;e_1^2 = e_1$, $\;\;e_1 f_3 = \frac{1}{2}f_3$. &$5$ & NA \\
$(1,3)_3$:&$\;\;e_1^2 = e_1$, $\;\;e_1 f_2 = \frac{1}{2}f_2$, $\;\;e_1 f_3 = \frac{1}{2}f_3$.&$5$ &NA\\
$(1,3)_4$:&$\;\;e_1^2 = e_1$, $\;\;e_1 f_2 = \frac{1}{2}f_2$, $\;\;e_1 f_3 = \frac{1}{2}f_3$, $\;\;f_2 f_3 = e_1$. &$4$&NA\\
$(1,3)_5$:&$\;\;e_1^2 = e_1$, $\;\;e_1 f_1 = \frac{1}{2}f_1$, $\;\;e_1 f_2 = \frac{1}{2}f_2$, $\;\;e_1 f_3 = \frac{1}{2} f_3$. &$9$&NA\\
$(1,3)_6$:&$\;\;e_1^2 = e_1$, $\;\;e_1 f_3 = f_3$. &$5$&A \\
$(1,3)_{7}$:&$\;\;e_1^2 = e_1$, $\;\;e_1 f_2 = \frac{1}{2}f_2$, $\;\;e_1 f_3 = f_3$. &$3$&NA \\
$(1,3)_{8}$:&$\;\;e_1^2 = e_1$, $\;\;e_1 f_1 = \frac{1}{2}f_1$, $\;\;e_1 f_2 = \frac{1}{2}f_2$, $\;\;e_1 f_3 = f_3$. &$5$ &NA\\
$(1,3)_{9}$:&$\;\;e_1^2 = e_1$, $\;\;e_1 f_2 = f_2$, $\;\;e_1 f_3 = f_3$. &$5$ &A\\
$(1,3)_{10}$:&$\;\;e_1^2 = e_1$, $\;\;e_1 f_2 = f_2$, $\;\;e_1 f_3 = f_3$, $\;\;f_2 f_3 = e_1$. &$4$ &NA\\
$(1,3)_{11}$:&$\;\;e_1^2 = e_1$, $\;\;e_1 f_1 = \frac{1}{2} f_1$, $\;\;e_1 f_2 = f_2$, $\;\;e_1 f_3 = f_3$. &$5$ &NA\\
$(1,3)_{12}$:&$\;\;e_1^2 = e_1$, $\;\;e_1 f_1 = f_1$, $\;\;e_1 f_2 = f_2$, $\;\;e_1 f_3 = f_3$. &$9$&A\\
$(1,3)_{13}$:&$\;\;e_1^2 = e_1$, $\;\;e_1 f_1 = f_1$, $\;\;e_1 f_2 = f_2$, $\;\;e_1 f_3 = f_3$, $\;\;f_1 f_2 = e_1$. &$6$&NA\\
$(1,3)^N_{14}$:&$\;\;e_1 f_2 = f_1$. &$6$&A\\
$(1,3)^N_{15}$:&$\;\;e_1 f_2 = f_1$, $\;\;f_2 f_3 = e_1$. &$5$&NA\\
$(1,3)_{16}$:&$\;\;e_1 f_2 = f_1$, $\;\;f_1 f_3 = e_1$. &$3$&NA\\
$(1,3)_{17}$:&$\;\;e_1 f_2 = f_1$, $\;\;f_1 f_2 = e_1$. &$4$&NA\\
$(1,3)^N_{18}$:&$\;\;e_1 f_2 = f_1$, $\;\;e_1 f_3 = f_2$. &$4$&NA\\
$(1,3)^N_{19}$:&$\;\;f_1 f_2 = e_1$. &$7$&A\\
$(1,3)_{20}^N$:&$\;\;e_1^2=0$, $\;\;e_1 f_j = f_i f_j = 0$. &$10$&A\\
\end{longtable}

 In order to determine the  associated Hasse diagram, 
we begin  by using (i) and (vi) in Lemma \ref{lemma:invariants}.
Then, Table \ref{Non_degenerations_1_3}  
 ensures 
the non-existence of deformations $(1,3)_i \not\rightarrow (1,3)_j$. We denote it briefly by $i \not\rightarrow j$. 
Table \ref{bases-(1, 3)} gives all possible essential deformations 
between Jordan superalgebras of type $(1,3)$ and the other deformations can be obtained by transitivity.

\begin{longtable}[H]{| l | c |}
\caption{ \label{Non_degenerations_1_3} Non-deformations between Jordan superalgebras of type $(1, 3)$}\\
\hline 
\multicolumn{1}{|c|}{$\mathcal{J} \not\rightarrow \mathcal{J}^\prime $ } & Reason \\
\hline
\endhead
$ i  \not\rightarrow  j$, for $i \in \{14, \cdots, 20\}$ and 
  $j \in \{1, \cdots, 13\}$. &  $\mathcal{J}_0  \not\rightarrow  \mathcal{J}^\prime_0$\\
%%
%\hline
\hline
$i \not \rightarrow 13, 19$,  for $i \in \{2,3,8, 11\}$; $\;\; i \not \rightarrow 19$,  for $i \in \{6, 9, 14\}$;  
& $ a(\mathcal{J}) \not \rightarrow  a(\mathcal{J}^\prime)$\\
$7 \not \rightarrow 4, 10, 13, 17, 19$; $\;\;18 \not \rightarrow  15, 19$. 
&\\
\hline
%\hline
%%
%%
$4 \not \rightarrow 1, 2,  5, 6, 8, 9, 11, 12, 13$; $\;\;10 \not \rightarrow  1, 2, 3, 5, 6,  8, 11,12, 13$;
& $F(\mathcal{J}) \not \rightarrow F(\mathcal{\mathcal{J}^\prime})$\\
$13 \not \rightarrow 1, 5$;  $\;\; 16 \not \rightarrow  18$;  $\;\; 17, 18  \not \rightarrow  9$. &\\
\hline
%\hline
%%
%%
$i \not \rightarrow 1,5,12$, for $i \in \{2, 3, 8, 11, 19\}$; 
$\;\; i \not \rightarrow 1,12$, for $i \in \{6, 9\}$; 
& $\mathcal{J} \not \rightarrow \mathcal{J}^\prime$\\ 
$7 \not \rightarrow 1, 2, 3, 5, 6, 8, 9, 11,12, 15$. & as algebras\\
\hline
\end{longtable}

\begin{longtable}[H]{|l|llll |}
\caption{\label{bases-(1, 3)}Deformations between Jordan superalgebras of type $(1,3)$}
\\
\hline
$\mathcal{J}  \rightarrow \mathcal{J}^\prime$  & Change of basis &  & & \\
\hline
\endhead
$(1,3)_{2} \rightarrow (1,3)_{14}^N$ & $E_1 =te_1 $, & $F_1 = \frac{t}{2} f_3  $, & $F_2 = f_1+f_3 $, & $F_3 = f_2$. \\
\hline
$(1,3)_{3} \rightarrow (1,3)_{14}^N$ & $E_1 = t e_1$, & $F_1 = t f_1$, & $F_2 = -2 f_1 +  f_2$, & $F_3 =  f_3$. \\
\hline
$(1,3)_{4} \rightarrow (1,3)_{3}$ & $E_1 =e_1 $, & $F_1 =f_1$, & $F_2 =f_2 $, & $F_3 =tf_3  $. \\
\hline
$(1,3)_{4} \rightarrow (1,3)_{15}^N$ & $E_1 = t e_1$, & $F_1 = t f_1$, & $F_2 = -2 f_1 +  f_2$, & $F_3 =  t f_3$. \\
\hline
$(1,3)_{6} \rightarrow (1,3)_{14}^N$ & $E_1 = te_1$, & $F_1 =tf_1  $, & $F_2 =- f_1+f_3 $, & $F_3 =f_2$. \\
\hline
$(1,3)_{7} \rightarrow (1,3)_{18}^N$ & $E_1 = t e_1$, &   
$F_1 = \frac{1}{4} t^2 f_2 + t^2 f_3$,   
&$F_2  =  \frac{1}{2}  t f_2 + t f_3$, 
& $F_3 = f_1+ f_2 + f_3$. 
 \\
\hline
$(1,3)_{8} \rightarrow (1,3)_{14}^N$ & $E_1 = t e_1$, & $F_1 =  \frac{1}{2} t f_3$, & $F_2 = f_2 +  f_3$, & $F_3 = f_1$. \\
\hline
$(1,3)_{9} \rightarrow (1,3)_{14}^N$ & $E_1  = t e_1$, & $F_1  = t  f_2$, & $F_2  =  f_1 + f_2$, & $F_3  = f_3$. \\
\hline
$(1,3)_{10} \rightarrow (1,3)_{9}$ & $E_1 =e_1 $, & $F_1 =f_1$, & $F_2 =f_2 $, & $F_3 =tf_3 $. \\
\hline
$(1,3)_{10} \rightarrow (1,3)_{15}^N$ & $E_1 = t e_1$, & $F_1 = t^2  f_1$, & $F_2  =  - tf_1+t f_2$, & $F_3  = f_3$. \\
\hline
$(1,3)_{11} \rightarrow (1,3)_{14}^N$ & $E_1 = t e_1$, &  $F_1 = - \frac{1}{2} t f_1$, & $F_2 =f_1+ f_2+f_3$, & $F_3 = f_3$. \\
\hline
$(1,3)_{13} \rightarrow (1,3)_{12}$ & $E_1 = e_1$, & $F_1 = t f_1$, & $F_2 = f_2$, & $F_3 =  f_3$. \\
\hline
$(1,3)_{13} \rightarrow (1,3)_{19}^N$ & $E_1  = t e_1$, & $F_1  = t f_1$, & $F_2  = f_2$, & $F_3  = f_3$.  \\
\hline
$(1,3)_{15}^N \rightarrow (1,3)_{14}^N$ & $E_1  = e_1$, & $F_1  = t f_1$, & $F_2  = t f_2$, & $F_3  = f_3$. \\
\hline
$(1,3)_{15}^N \rightarrow (1,3)_{19}^N$ & $E_1  = t e_1$, &  $F_1  = -t f_3$, & $F_2  = f_2$, & $F_3  = f_1$. \\
\hline
$(1,3)_{16} \rightarrow (1,3)_{17}$ & $E_1= e_1$, & $F_1= f_1$, & $F_2 = f_1 + f_2 + f_3$, &  $F_3= t f_2$. \\
\hline
$(1,3)_{17} \rightarrow (1,3)_{15}^N$ & $E_1 = t^2 e_1$, & $F_1 = t^2 f_3$, & $F_2 = f_1 + t f_2$, &  $F_3 = t^2 f_2 + f_3$. \\ 
\hline
$(1,3)_{18}^N \rightarrow (1,3)_{14}^N$ & $E_1  = t e_1$, & $F_1  = t f_1$, & $F_2  =  f_2$, & $F_3 =  f_3$.\\
\hline
\end{longtable}

According to the above information, we  deduce 
the principal result of this section.   

\begin{theorem}\label{theorem:irreducible_components_1_3}There exist exactly $11$ rigid Jordan superalgebras of type $(1,3)$ whose Zariski closures of their orbits are the irreducible components of the variety $\mathcal{JS}^{(1,3)}$. The rigid superalgebras are:
$$
 (1,3)_i  \quad  \text{for} \quad 
i \in 
\{1,2,4,5,6,7,8,10, 11, 13, 16
\}
$$
and the irreducible components are 
\begin{align*}
    \overline{ (1,3)_{1}^G }  & =  \{  (1,3)_{1},  (1,3)_{20}^N  \} \\
\overline{ (1,3)_{2}^G }  & =  \{  (1,3)_{2},  (1,3)_{14}^N,  (1,3)_{20}^N  \}    \\
\overline{ (1,3)_{4}^G }  & =  \{(1,3)_{3},  (1,3)_{4},     (1,3)_{14}^N, (1,3)_{15}^N,   (1,3)_{19}^N,  (1,3)_{20}^N \}    \\
\overline{ (1,3)_{5}^G }  & =  \{  (1,3)_{5},  (1,3)_{20}^N   \}    \\
\overline{ (1,3)_{6}^G }  & =  \{  (1,3)_{6},  (1,3)_{14}^N,  (1,3)_{20}^N   \}    \\
\overline{ (1,3)_{7}^G }  & =  \{  (1,3)_{7},  (1,3)_{14}^N,  (1,3)_{18}^N,   (1,3)_{20}^N \}    \\
\overline{ (1,3)_{8}^G }  & =  \{  (1,3)_{8},  (1,3)_{14}^N,  (1,3)_{20}^N  \}    \\
\overline{ (1,3)_{10}^G }  & =  \{  (1,3)_{9}, (1,3)_{10},   (1,3)_{14}^N,   (1,3)_{15}^N,  (1,3)_{19}^N,  (1,3)_{20}^N \}    \\
\overline{ (1,3)_{11}^G }  & =  \{  (1,3)_{11},  (1,3)_{14}^N,  (1,3)_{20}^N   \}    \\
\overline{ (1,3)_{13}^G}  & =  \{(1,3)_{12},    (1,3)_{13},  (1,3)_{19}^N,    (1,3)_{20}^N  \}    \\
\overline{ (1,3)_{16}^G }  & =  \{(1,3)_{14}^N,  (1,3)_{15}^N,   (1,3)_{16},  (1,3)_{17},   (1,3)_{19}^N,  (1,3)_{20}^N \}
\end{align*}
\end{theorem}

Moreover, in \cite{VergneConj}, M. Vergne conjectured that there is no complex $n$-dimensional nilpotent Lie algebra which is rigid in the variety of all $n$-dimensional complex Lie algebras. Similar conjectures could  potentially be  applied to other varieties. In particular, no examples of nilpotent or solvable rigid Jordan algebras are known so far. As a consequence of Theorem \ref{theorem:irreducible_components_1_3} we prove that the conjecture does not hold in the superalgebras case. Indeed, the superalgebra $(1,3)_{16}$ is a four-dimensional solvable rigid Jordan superalgebra.

The irreducible components of $\mathcal{JS}^{(1,3)}$ are represented  in Figure \ref{grafica_1_3}. In Hasse diagrams we adopt the following notation: the blue color indicates an associative superalgebra, and a square represents a nilpotent superalgebra. Furthermore, for abbreviation, we let $i$ stands for either $(1,3)_i$ or $(1,3)_i^N$.

\begin{figure}[H]
  \centering
    \includegraphics[width=8cm , height=8cm]{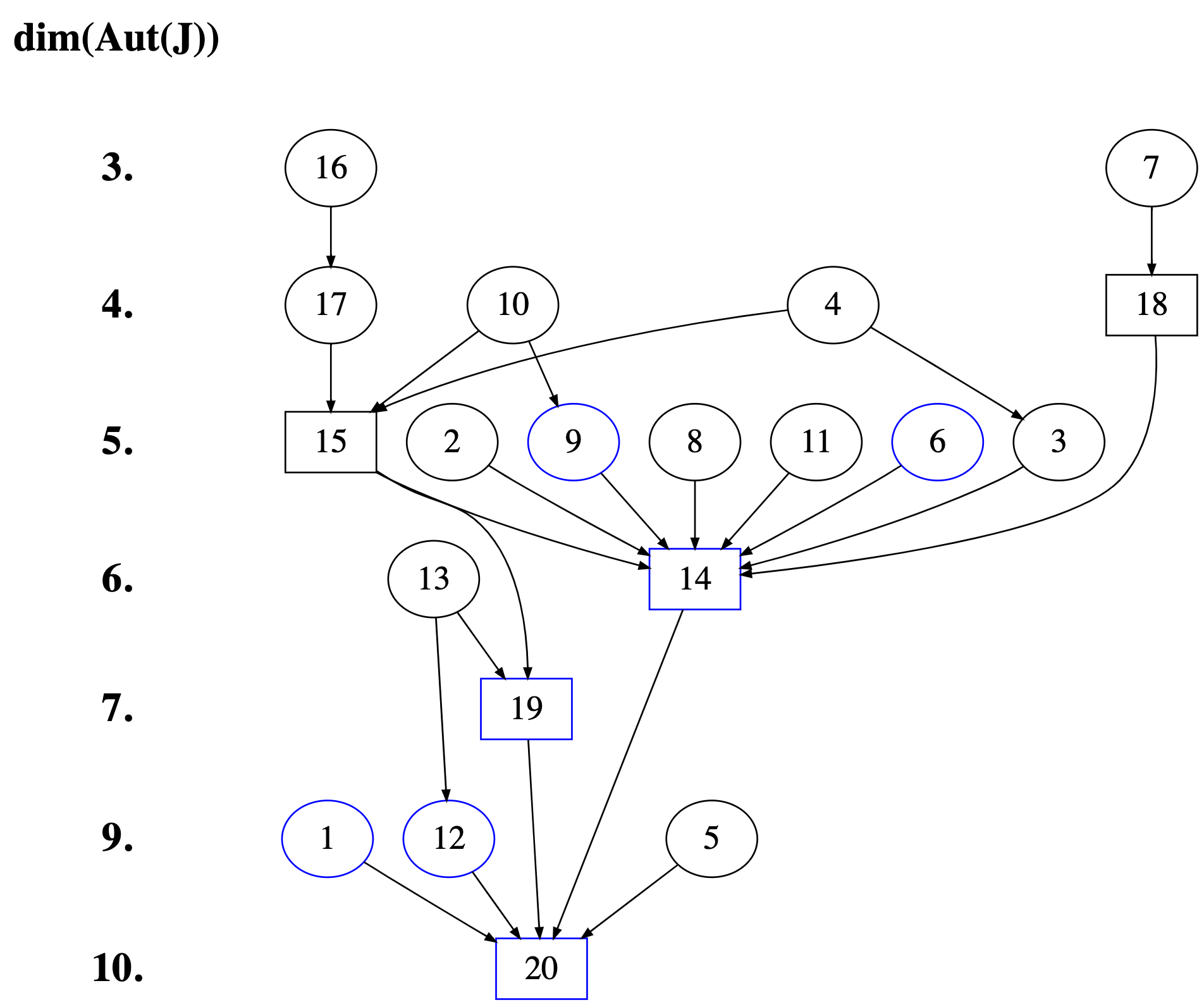}
  \caption{
  \label{grafica_1_3}
  Hasse diagram of deformations for Jordan  superalgebras of type $(1,3)$
  }
\end{figure}

 Notice that the conditions of associativity and nilpotency are determined by polynomial equations on the structure constants, so the 
 $\mathcal{ASC}^{(m,n)} \subset \mathcal{JS}^{(m,n)}$ and $\mathcal{NJS}^{(m,n)}   \subset \mathcal{JS}^{(m,n)}$ are affine subvarieties of associative supercommutative  and nilpotent Jordan superalgebras of type 
$(m,n)$, respectively. As a direct consequence, we achieve the geometric classification of these varieties.

\begin{corollary}
The subvariety $\mathcal{ASC}^{(1,3)} \subset \mathcal{JS}^{(1,3)}$ 
has $5$ irreducible components given by 
\begin{align*}
    \overline{ (1,3)_{4}^G } & =\{  (1,3)_{4},  (1,3)_{20}^N  \}\\
\overline{ (1,3)_{6}^G }  & =\{  (1,3)_{6},  (1,3)_{14}^N,  (1,3)_{20}^N  \}   \\
\overline{ (1,3)_{9}^G }  & =\{  (1,3)_{9},  (1,3)_{14}^N,  (1,3)_{20}^N  \}   \\
\overline{ (1,3)_{12}^G }  & =\{  (1,3)_{12},   (1,3)_{20}^N  \}    \\
\overline{ ((1,3)_{19}^N)^G }  & =\{  (1,3)_{19}^N,   (1,3)_{20}^N  \}  
\end{align*}
\end{corollary}

\begin{corollary}
The subvariety $\mathcal{NJS}^{(m,n)} \subset \mathcal{JS}^{(1,3)}$ 
has $2$ irreducible components given by 
\begin{align*}
\overline{ ((1,3)^N_{15})^G }  & =  \{ (1,3)_{14}^N, (1,3)_{15}^N,(1,3)_{19}^N, (1,3)_{20}^N  \}    \\
\overline{ ((1,3)_{18}^N)^G }  & = \{  (1,3)_{14}^N,  (1,3)_{18}^N, (1,3)_{20}^N  \}   
\end{align*}
\end{corollary}

\section{Jordan Superalgebras of Type $(3,1)$}

%%
%%  Jordan Superalgebras of_1 type (3,1)
%%

% para esta seccion usamos los resultados dados en los articulos de Irina e Eugenia,  ademas de resultados mas finos que estan es sus tesis de doctorado, que citamos 

In this section, we proceed with the study  of the variety $\mathcal{JS}^{(3,1)}$.  
There exist, up to isomorphism, $60$  Jordan superalgebras of type $(3,1)$ (see \cite{SuperPowerAssoc4}, Table 16, p. 1630).  Table \ref{table:3_1} %and \ref{table:3_1_continuation} 
gives representatives for  isomorphism classes and some additional useful information, 
 namely the dimensions of the automorphism groups. As in the previous section, the labels ``A", ``NA" and ``N" indicate that the superalgebra is associative, non-associative and nilpotent, respectively.

%and we write "A" if the superalgebra is associative and "NA" otherwise. 
%It is important to point out that a {\color{red}simple comparision of the results given in Tables  \ref{table:3_1} and \ref{table:3_1_continuation}  with \cite{Eugenia} show that} $\dim(\Aut(J)) = \dim(\Aut(J_0))+1$, for all $J \in \mathcal{JS}^{(3,1)}$.  

\begin{longtable}[H]{rlcc}
\caption{\label{table:3_1}Jordan superalgebras of type $(3,1)$.}\\
Label  & Multiplication table & $\dim(\Aut(\mathcal{J})) $ & \\
\hline
\endhead
$(3,1)_1$:&$\;\;e_1^2 = e_1$, $\;\;e_2^2 = e_2$, $\;\;e_1e_3 = e_3$, $\;\;e_1f_1 = \frac{1}{2}f_1$, $\;\;e_2f_1 = \frac{1}{2}f_1$. & $2$&NA \\
$(3,1)_2$:&$\;\;e_1^2 = e_1$, $\;\;e_2^2 = e_3$, $\;\;e_1e_2 = e_2$, $\;\;e_1e_3 = e_3$, $\;\;e_1f_1 = f_1$.
& $3$&A\\
$(3,1)_3$:&$\;\;e_1^2 = e_1$, $\;\;e_2^2 = e_3$, $\;\;e_1e_2 = e_2$, $\;\;e_1e_3 = e_3$, $\;\;e_1f_1 = \frac{1}{2}f_1$. 
& $3$&NA\\
$(3,1)_4$:&$\;\;e_1^2 = e_1$, $\;\;e_1e_2 = e_2$, $\;\;e_1e_3 = e_3$, $\;\;e_1f_1 = f_1$.
&$5$&A\\
$(3,1)_5$:&$\;\;e_1^2 = e_1$, $\;\;e_1e_2 = e_2$, $\;\;e_1e_3 = e_3$, $\;\;e_1f_1 = \frac{1}{2}f_1$.
& $5$&
	NA \\
$(3,1)_6$:&$\;\;e_1^2 = e_1$, $\;\;e_2^2 = e_1+e_3$, $\;\;e_3^2 = e_3$, $\;\;e_1e_2 = \frac{1}{2}e_2$, $\;\;e_2e_3 = \frac{1}{2}e_2$, & $2$&NA \\
 & $\;\;e_1f_1 = \frac{1}{2}f_1$, $\;\;e_3f_1 = \frac{1}{2}f_1$.  && \\
$(3,1)_{7}$:&$\;\;e_1^2 = e_1$, $\;\;e_1e_2 = \frac{1}{2}e_2$, $\;\;e_1e_3 = e_3$, $\;\;e_1f_1 = f_1$. 
&$4$& NA \\
$(3,1)_{8}$:&$\;\;e_1^2 = e_1$, $\;\;e_1e_2 = \frac{1}{2}e_2$, $\;\;e_1e_3 = e_3$, $\;\;e_1f_1 = \frac{1}{2}f_1$. 
&$4$& NA\\
$(3,1)_{9}$:&$\;\;e_1^2 = e_1$, $\;\;e_1e_2 = \frac{1}{2}e_2$, $\;\;e_1e_3 = \frac{1}{2}e_3$, $\;\;e_1f_1 = f_1$. 
&$7$ & NA \\
$(3,1)_{10}$:&$\;\;e_1^2 = e_1$, $\;\;e_1e_2 = \frac{1}{2}e_2$, $\;\;e_1e_3 = \frac{1}{2}e_3$, $\;\;e_1f_1 = \frac{1}{2}f_1$.  &  $7$&NA\\
$(3,1)_{11}$:&$\;\;e_1^2 = e_1$, $\;\;e_2^2 = e_3$, $\;\;e_1e_2 = \frac{1}{2}e_2$, $\;\;e_1f_1 = f_1$.  
& $3$&NA\\
$(3,1)_{12}$:&$\;\;e_1^2 = e_1$, $\;\;e_2^2 = e_3$, $\;\;e_1e_2 = \frac{1}{2}e_2$, $\;\;e_1f_1 = \frac{1}{2}f_1$. 
& $3$&NA\\
$(3,1)_{13}$:&$\;\;e_1^2 = e_1$, $\;\;e_2^2 = e_3$, $\;\;e_1e_2 = \frac{1}{2}e_2$, $\;\;e_1e_3 = e_3$, $\;\;e_1f_1 = f_1$. & $3$ &NA \\
$(3,1)_{14}$:&$\;\;e_1^2 = e_1$, $\;\;e_2^2 = e_3$, $\;\;e_1e_2 = \frac{1}{2}e_2$, $\;\;e_1e_3 = e_3$, $\;\;e_1f_1 = \frac{1}{2}f_1$.  & $3$ & NA\\
$(3,1)_{15}$:&$\;\;e_1^2 = e_1$, $\;\;e_2^2 = e_2$, $\;\;e_1e_3 = \frac{1}{2}e_3$, $\;\;e_2e_3 = \frac{1}{2}e_3$, $\;\;e_2f_1 = \frac{1}{2}f_1$. & $3$ & NA\\
$(3,1)_{16}$:&$\;\;e_1^2 = e_1$, $\;\;e_2^2 = e_2$, $\;\;e_1e_3 = \frac{1}{2}e_3$, $\;\;e_2e_3 = \frac{1}{2}e_3$, $\;\;e_2f_1 = f_1$.&$3$ & NA\\
$(3,1)_{17}$:&$\;\;e_1^2 = e_1$, $\;\;e_2^2 = e_2$, $\;\;e_1e_3 = \frac{1}{2}e_3$, $\;\;e_2e_3 = \frac{1}{2}e_3$, $\;\;e_1f_1 = \frac{1}{2}f_1$.  & $3$ & NA\\
& $\;\;e_2f_1 = \frac{1}{2}f_1$.&&\\
$(3,1)_{18}$:&$\;\;e_1^2 = e_1$, $\;\;e_2^2 = e_2$, $\;\;e_2e_3 = \frac{1}{2}e_3$, $\;\;e_1f_1 = \frac{1}{2}f_1$, $\;\;e_2f_1 = \frac{1}{2}f_1$. 
& $3$& NA \\
$(3,1)_{19}$:& $\;\;e_1^2=e_1$, $\;\;e_2^2=e_2$, $\;\;e_3^2=e_3$. &$1$&A\\
$(3,1)_{20}$:& $\;\;e_1^2=e_1$, $\;\;e_2^2=e_2$, $\;\;e_3^2=e_3$, $\;\;e_3f_1=f_1$. & $1$&A\\
$(3,1)_{21}$:& $\;\;e_1^2=e_1$, $\;\;e_2^2=e_2$, $\;\;e_3^2=e_3$, $\;\;e_3f_1= \frac{1}{2}f_1$. &$1$& NA\\
$(3,1)_{22}$:& $\;\;e_1^2=e_1$, $\;\;e_2^2=e_2$, $\;\;e_3^2=e_3$, $\;\;e_1f_1= \frac{1}{2}f_1$, $\;\;e_2f_1= \frac{1}{2}f_1$. & $1$& NA\\
$(3,1)_{23}$:& $\;\;e_1^2=e_1$, $\;\;e_2^2=e_2$, $\;\;e_1e_3=e_3$. &$2$&A\\
$(3,1)_{24}$:& $\;\;e_1^2=e_1$, $\;\;e_2^2=e_2$, $\;\;e_1e_3=e_3$, $\;\;e_1f_1=f_1$. & $2$&A\\
$(3,1)_{25}$:& $\;\;e_1^2=e_1$, $\;\;e_2^2=e_2$, $\;\;e_1e_3=e_3$, $\;\;e_1f_1=\frac{1}{2}f_1$.  & $2$& NA \\
$(3,1)_{26}$:& $\;\;e_1^2=e_1$, $\;\;e_2^2=e_2$, $\;\;e_1e_3=e_3$, $\;\;e_2f_1=f_1$. &$2$ &A\\
$(3,1)_{27}$:& $\;\;e_1^2=e_1$, $\;\;e_2^2=e_2$, $\;\;e_1e_3=e_3$, $\;\;e_2f_1=\frac{1}{2}f_1$.  & $2$&NA \\
$(3,1)_{28}$:& $\;\;e_1^2=e_1$, $\;\;e_2^2=e_3$, $\;\;e_1e_2=e_2$, $\;\;e_1e_3=e_3$. &$3$ &A\\
$(3,1)_{29}$:& $\;\;e_1^2=e_1$,  $\;\;e_1e_2=e_2$, $\;\;e_1e_3=e_3$. & $5$&A \\
$(3,1)_{30}$:& $\;\;e_1^2=e_1$,  $\;\;e_2^2=e_2$. &$2$ &A \\
$(3,1)_{31}$:& $\;\;e_1^2=e_1$,  $\;\;e_2^2=e_2$, $\;\;e_1f_1=f_1$.  &$2$  &A\\
$(3,1)_{32}$:& $\;\;e_1^2=e_1$,  $\;\;e_2^2=e_2$, $\;\;e_1f_1=\frac{1}{2}f_1$.  &$2$ &NA \\
$(3,1)_{33}$:& $\;\;e_1^2=e_1$,  $\;\;e_2^2=e_2$, $\;\;e_1f_1=\frac{1}{2}f_1$, $\;\;e_2f_1=\frac{1}{2}f_1$. &  $2$ & NA \\
$(3,1)_{34}$:& $\;\;e_1^2=e_1$,  $\;\;e_1e_2=e_2$.  &$3$ & A \\
$(3,1)_{35}$:& $\;\;e_1^2=e_1$,  $\;\;e_1e_2=e_2$, $\;\;e_1f_1=f_1$. &$3$ &A  \\
$(3,1)_{36}$:& $\;\;e_1^2=e_1$,  $\;\;e_1e_2=e_2$, $\;\;e_1f_1=\frac{1}{2}f_1$. &$3$  & NA \\
$(3,1)_{37}$:& $\;\;e_1^2=e_1$,  $\;\;e_2^2=e_3$.&$3$ & A \\
$(3,1)_{38}$:& $\;\;e_1^2=e_1$,  $\;\;e_2^2=e_3$, $\;\;e_1f_1=f_1$.& $3$ &A \\
$(3,1)_{39}$:& $\;\;e_1^2=e_1$,  $\;\;e_2^2=e_3$, $\;\;e_1f_1=\frac{1}{2}f_1$. &$3$ &NA \\
$(3,1)_{40}^N$:& $\;\;e_1^2=e_2$,  $\;\;e_1e_2=e_3$. &$4$ & A \\
$(3,1)_{41}$:& $\;\;e_1^2=e_1$. & $5$&A\\
$(3,1)_{42}$:& $\;\;e_1^2=e_1$,  $\;\;e_1f_1=f_1$.& $5$&A\\
$(3,1)_{43}$:& $\;\;e_1^2=e_1$,  $\;\;e_1f_1=\frac{1}{2}f_1$. & $5$& NA\\
$(3,1)_{44}^N$:& $\;\;e_1e_2=e_3$. & $5$&A \\
$(3,1)_{45}^N$:& $\;\;e_1^2=e_2$. &$6$ & A \\
$(3,1)_{46}$:& $\;\;e_1^2=e_1$, $\;\; e_2^2=e_1+e_3$, $\;\;e_3^2=e_3$, $\;\;  e_1e_2=\frac{1}{2}e_2$, $\;\;e_2e_3=\frac{1}{2}e_2$. & $2$ &NA \\
%&$\;\;e_2e_3=\frac{1}{2}e_2$.&& \\
$(3,1)_{47}$:& $\;\;e_1^2=e_1$, $\;\;  e_1e_2=\frac{1}{2}e_2$, $\;\;e_1e_3=e_3$. & $4$ &NA \\
$(3,1)_{48}$:&$\;\;e_1^2=e_1$, $\;\;  e_1e_2=\frac{1}{2}e_2$, $\;\;e_1e_3=\frac{1}{2}e_3$.  &$7$  & NA \\
$(3,1)_{49}$:& $\;\;e_1^2=e_1$, $\;\;  e_2^2=e_3$, $\;\;e_1e_2=\frac{1}{2}e_2$. &$3$ & NA  \\
$(3,1)_{50}$:&$\;\;e_1^2=e_1$, $\;\;  e_2^2=e_3$, $\;\;e_1e_2=\frac{1}{2}e_2$, $\;\;e_1e_3=e_3$. &$3$ & NA \\
$(3,1)_{51}$:& $\;\;e_1^2=e_1$, $\;\;  e_2^2=e_2$, $\;\;e_1e_3=\frac{1}{2}e_3$, $\;\;e_2e_3=\frac{1}{2}e_3$. &$3$  & NA \\
$(3,1)_{52}$:& $\;\;e_1^2=e_1$, $\;\;  e_2^2=e_2$, $\;\;e_2e_3=\frac{1}{2}e_3$. & $3$ & NA  \\
$(3,1)_{53}$:& $\;\;e_1^2=e_1$, $\;\;  e_2^2=e_2$, $\;\;e_2e_3=\frac{1}{2}e_3$, $\;\;e_1f_1=f_1$. &$3$ & NA\\
$(3,1)_{54}$:& $\;\;e_1^2=e_1$, $\;\;  e_2^2=e_2$, $\;\;e_2e_3=\frac{1}{2}e_3$, $\;\;e_1f_1=\frac{1}{2}f_1$. &$3$& NA \\
$(3,1)_{55}$:& $\;\;e_1^2=e_1$, $\;\;  e_2^2=e_2$, $\;\;e_2e_3=\frac{1}{2}e_3$,$\;\;e_2f_1=f_1$. & $3$ & NA \\
$(3,1)_{56}$:& $\;\;e_1^2=e_1$, $\;\;  e_2^2=e_2$, $\;\;e_2e_3=\frac{1}{2}e_3$, $\;\;e_2f_1=\frac{1}{2}f_1$. &$3$ & NA\\
$(3,1)_{57}$:& $\;\;e_2^2=e_2$, $\;\;e_2e_3=\frac{1}{2}e_3$. &  $4$ & NA\\
$(3,1)_{58}$:& $\;\;e_2^2=e_2$, $\;\;e_2e_3=\frac{1}{2}e_3$, $\;\;e_2f_1=f_1$. & $4$ & NA \\
$(3,1)_{59}$:& $\;\;e_2^2=e_2$, $\;\;e_2e_3=\frac{1}{2}e_3$, $\;\;e_2f_1=\frac{1}{2}f_1$.  &$4$ &NA \\
$(3,1)_{60}^N$:& $\;\;e_ie_j=0$, $e_if_1=0$, $\;\;i,j \in \{1,2,3\}$. &$10$&A\\
\hline
\end{longtable}

As in the previous section, Table \ref{table:non_degenerations_3_1} ensures
the non-existence of deformations $(3,1)_i\not\rightarrow (3,1)_j$. Recall that we first use items (i) and (vi) of Lemma \ref{lemma:invariants} together with the information in Table \ref{table:3_1}.

\begin{longtable}[H]{| l | c |}
\caption{\label{table:non_degenerations_3_1}Non-deformations between Jordan Superalgebras of type  $(3,1)$}
\\
\hline
\multicolumn{1}{|c|}{$\mathcal{J} \not\rightarrow \mathcal{J}^\prime $ } & Reason \\
\hline
\endhead
$ i  \not\rightarrow 7, \dots, 18, 47, \dots,  50$, for $i \in \{1, 25, 27\}$; $\;\; i  \not\rightarrow 41, 42 $, for $i \in \{ 2, 28\}$; &  \\
$3 \not\rightarrow 7, \dots, 10, 41, 42,  43, 47, 48, 57, 58, 59 $;  & \\
$i  \not\rightarrow 7 , \dots, 10, 13, 14, 18, 34 , \dots, 39, 41, 42, 43, 47, 48, 50, 52, \dots, 56$,  &  \\
 for $i \in \{ 6, 46\}$;   &  \\
$ i  \not\rightarrow 4, 5, 9, 10, 29, 41, 42, 43, 48$, for $i \in \{ 7, 8, 47, 57, 58, 59 \}$; & \\
$i  \not\rightarrow  4,  \dots, \hat 6, \dots, 10,  29, 41, 42, 43, 47, 48$,  for $i \in \{11, 12, 49\}$;  & \\
$i  \not\rightarrow  4, 5, 9, 10, 29, 41, 42, 43, 48, 57, 58, 59$, for $i \in \{13, 14, 50\}$;  & $\mathcal{J}_0  \not\rightarrow \mathcal{J}^\prime_0$ \\
$i  \not\rightarrow  7, \dots, 10, 41, 42, 43, 47, 48$, for $i \in \{15, 16, 17, 51\}$;   & \\
$i  \not\rightarrow 4, 5, 9, 10,  29, 48, 57, 58, 59$,  for $i \in \{18,  52, \dots, 56 \}$; & \\
$i \not\rightarrow 6, \dots, 18, 46, \dots, 59$, for $i \in  \{21, 22\}$; & \\
$i  \not\rightarrow 11, \dots, 18, 49, \dots, 56$, for $i \in  \{23, 24, 26\}$; & \\
$i  \not\rightarrow 2, 4, 11, \dots, 18, 28, 29, 49 , \dots, 56$, for $i \in  \{30, 31\}$; & \\
$i  \not\rightarrow 2, \dots,  \hat 6, \dots, 18, 28, 29, 47 , \dots, 59$,  for $i \in  \{32, 33\} $; &   \\
$i  \not\rightarrow 9, 10, 48$,  for $i \in  \{5, 43\} $;&   \\
$i   \not\rightarrow 4, 29, 41, 42 $,  for $i \in  \{34, 35, 40\}$; &   \\
$36  \not\rightarrow  4, \dots, \hat 6 , \dots, 10, 29, 41, 42, 43, 47, 48, 57, 58, 59 $;  & \\
\hline
$i   \not\rightarrow 4, 29$,  for $i \in  \{37, 38\}$;  &  $\mathcal{J}_0  \not\rightarrow \mathcal{J}^\prime_0$\\
$39  \not\rightarrow 4, \dots, \hat 6 , \dots,  10, 29, 47, 48, 57, 58, 59 $.  & \\
%%%%%%%
\hline
$19  \not\rightarrow  2, 4, 24, 26,  31, 35, 38, 42$; $\;\; 23  \not \rightarrow 2, 4, 35, 38, 42 $; 
 & $  \dim((\mathcal{J}^2)_1) <     \dim((\mathcal{J}^\prime)^2)_1)$   \\
$28 \not \rightarrow 4 $; $\;\; 30  \not \rightarrow 35, 38,  42$; $\;\;37 \not \rightarrow 42$;  $\;\;46  \not \rightarrow 2, 3$;  $\;\;52 \not \rightarrow 43 $. & \\
\hline
$1 \not \rightarrow 3, 5, 28, 29, 34, 35, 37, 38, 41, 42 $; $\;\;  2 \not \rightarrow 29  $; $\;\;  3 \not \rightarrow  4, 29  $;   
 & \\
$6 \not \rightarrow  5,  11, 15,  16, 29, 49, 51, 57, 58$; $\;\; 11 \not \rightarrow 57, 59  $; $\;\; 12 \not \rightarrow 57, 58  $;   
& \\
$\;\; 13\not \rightarrow 8, 47 $; $ \;\; 14 \not \rightarrow 7, 47  $;  $\;\; 15 \not \rightarrow 4, 29, 58$;  $\;\; 16 \not \rightarrow 5, 29, 59 $;    
&\\
$\;\; 17 \not \rightarrow 5, 29, 57, 58$; $\;\; 18 \not \rightarrow 8,  41, 42,  47$;   & \\ 
$ \;\; 21 \not \rightarrow  1, 2,  4, 23, 24, 26, 28, 29, 31, 33, 35 $; 
&\\
$22 \not \rightarrow 3, 5, 23, 25, 26,  27, 30, 31, 34, 38,  42  $;  $\;\;24 \not \rightarrow  28, 34, 38, 42 $; &\\
$25 \not \rightarrow 2, 4, 28, 29, 34, 35, 38, 39, 42, 43  $; $\;\;26 \not \rightarrow  28, 29, 35, 37, 41$;
& $\mathcal{J} \not \rightarrow \mathcal{J}^\prime$\\
$27 \not \rightarrow  2, 4, 28, 29, 35 , \dots, 38, 41$; $\;\;32 \not \rightarrow  34, 35, 38, 42 $;
& as algebras \\
$33 \not \rightarrow 34, 36, 37, 38, 41, 42$; $\;\;46 \not \rightarrow  4, 5, 11, 12, 15, 16, 17$; $\;\; 38 \not \rightarrow 41$;  
& \\
$39 \not \rightarrow 41, 42 $;  $\;\;49 \not \rightarrow  58, 59 $;  $\;\;50 \not \rightarrow 7, 8 $;  $\;\; 51 \not \rightarrow 4, 5,  58, 59   $;  
& \\
$ 52 \not \rightarrow 7, 8, 42 $;  
$\;\; 53 \not \rightarrow 8, 41, 47$; 
$ \;\;54 \not \rightarrow 7, 41, 42, 47$; 
$ \;\;55 \not \rightarrow 8, 42, 43  $;  & \\
$56 \not \rightarrow 7, 42, 43, 47.$&\\
\hline
\end{longtable}

Table \ref{table:3_1_degenerations_first} gives all possible essential deformations
between two Jordan superalgebras of type $(3,1)$  and the other deformations can be obtained by transitivity.

\begin{longtable}[H]{|l|llll|}
\caption{\label{table:3_1_degenerations_first} Deformations between Jordan Superalgebras of type  $(3,1)$}
\\
\hline
$\mathcal{J}  \rightarrow \mathcal{J}^\prime$  & Change of basis &  & & \\
\hline
\endhead
$(3, 1)_{1} \rightarrow  (3, 1)_{2} $& $E_1=e_1+e_2$, &$E_2=te_2+e_3$,& $E_3= t^2e_2$, &$F_1=f_1$\\
\hline
$(3, 1)_{1} \rightarrow  (3, 1)_{36} $& $E_1=e_1$, &$E_2=e_3$,& $E_3= te_2$, &$F_1=\frac{1}{2}f_1$\\
\hline
$(3, 1)_{1} \rightarrow  (3, 1)_{39} $& $E_1=e_2$, &$E_2=te_1+e_3$,& $E_3= -t^2e_1$, &$F_1=f_1$\\
\hline
$(3, 1)_{2} \rightarrow  (3, 1)_{4} $& $E_1=e_1$, &$E_2=te_2$,& $E_3=  e_3$, &$F_1=f_1$\\
\hline
$(3, 1)_{2} \rightarrow  (3, 1)_{40}^N $& $E_1=te_1+e_2$, &$E_2=t^2e_1+2te_2+e_3$,& $E_3= t^3e_1+3t^2e_2+3te_3$, &$F_1=f_1$\\
\hline
$(3, 1)_{3} \rightarrow  (3, 1)_{5} $ & $E_1= e_1$, & $E_2= te_2$, & $E_3= e_3$, &$ F_1=f_1$\\
\hline
$(3, 1)_{3} \rightarrow  (3, 1)_{40}^N $& $E_1=te_1+e_2$, &$E_2=t^2e_1+2 te_2+e_3  $,& $E_3= t^3e_1
+3t^2 e_2+ 3te_3$, &$F_1=f_1$\\
\hline
$(3, 1)_{4} \rightarrow  (3, 1)_{45}^N $& $E_1=t(e_1+e_2)$, &$E_2=t^2e_2$,& $E_3=  t^2e_1+t^3e_2+e_3$, &$F_1=f_1$\\
\hline
$(3, 1)_{5} \rightarrow  (3, 1)_{45}^N $ & $E_1= t(e_1+e_2)$, & $E_2= t^2e_2$, & $E_3=t^2e_1+t^3e_2+ e_3$, &$ F_1=f_1$\\
\hline
$(3, 1)_{6} \rightarrow  (3, 1)_{12} $& $E_1=e_1$, &$E_2=te_2$,& $E_3= 
 t^2(e_1+ e_3)$, &$F_1=f_1$\\
\hline
$(3, 1)_{6} \rightarrow  (3, 1)_{17} $& $E_1=e_1$, &$E_2=e_3$,& $E_3= 
 te_2$, &$F_1=f_1$\\
\hline
$(3, 1)_{7} \rightarrow  (3, 1)_{44}^N $ & $E_1= e_2+e_3$, & $E_2= te_1$, & $E_3=\frac{1}{2}te_2+ t e_3$, &$ F_1=f_1$\\
\hline
$(3, 1)_{8} \rightarrow  (3, 1)_{44}^N $ & $E_1= e_2+e_3$, & $E_2= t e_1$, & $E_3= \frac{1}{2}t e_2+ t e_3$, &$ F_1=f_1$\\
\hline
$(3, 1)_{11} \rightarrow  (3, 1)_{58} $ & $E_1= e_3$, & $E_2= e_1$, & $E_3= te_2$, &$ F_1=f_1$\\
\hline
$(3, 1)_{12} \rightarrow  (3, 1)_{59} $ & $E_1= e_3$, & $E_2= e_1$, & $E_3= te_2$, &$ F_1=f_1$\\
\hline
$(3, 1)_{13} \rightarrow  (3, 1)_{7} $ & $E_1= e_1$, & $E_2= te_2$, & $E_3= e_3$, &$ F_1=f_1$\\
\hline
$(3, 1)_{14} \rightarrow  (3, 1)_{8} $ & $E_1= e_1$, & $E_2= t^2e_2$, & $E_3= t^2e_3$, &$ F_1=f_1$\\
\hline
$(3, 1)_{15} \rightarrow  (3, 1)_{5} $ & $E_1= e_1+e_2$, & $E_2= te_2$, & $E_3= e_3$, &$ F_1=f_1$\\
\hline
$(3, 1)_{15} \rightarrow  (3, 1)_{57} $ & $E_1= te_2$, & $E_2= e_1$, & $E_3=  e_3$, &$ F_1=f_1$\\
\hline
$(3, 1)_{15} \rightarrow  (3, 1)_{59} $ & $E_1= te_1$, & $E_2= e_2$, & $E_3= t e_3$, &$ F_1=f_1$\\
\hline
$(3, 1)_{16} \rightarrow  (3, 1)_{4} $ & $E_1= e_1+e_2$, & $E_2= te_2$, & $E_3= e_3$, &$ F_1=f_1$\\
\hline
$(3, 1)_{16} \rightarrow  (3, 1)_{57} $ & $E_1= te_2$, & $E_2= e_1$, & $E_3=  e_3$, &$ F_1=f_1$\\
\hline
$(3, 1)_{16} \rightarrow  (3, 1)_{58} $ & $E_1= te_1$, & $E_2= e_2$, & $E_3=  te_3$, &$ F_1=f_1$\\
\hline
$(3, 1)_{17} \rightarrow  (3, 1)_{4} $ & $E_1= e_1+e_2$, & $E_2= te_2$, & $E_3= e_3$, &$ F_1=f_1$\\
\hline
$(3, 1)_{17} \rightarrow  (3, 1)_{59} $ & $E_1= te_1$, & $E_2= e_2$, & $E_3=  t e_3$, &$ F_1=f_1$\\
\hline
$(3, 1)_{18} \rightarrow  (3, 1)_{7} $ & $E_1= e_1+e_2$, & $E_2= e_3$, & $E_3=  t e_2$, &$ F_1=f_1$\\
\hline
$(3, 1)_{18} \rightarrow  (3, 1)_{43} $ & $E_1= e_1$, & $E_2= te_2$, & $E_3= e_3$, &$ F_1=f_1$\\
\hline
$(3, 1)_{19} \rightarrow  (3, 1)_{23} $ & $E_1= e_1+e_3$, & $E_2=e_2$, & $E_3=te_3$, &$ F_1=f_1$\\
\hline
$(3, 1)_{19} \rightarrow  (3, 1)_{30} $ & $E_1= e_1$, & $E_2=e_2$, & $E_3=te_3$, &$ F_1=f_1$\\
\hline
$(3, 1)_{20} \rightarrow  (3, 1)_{24} $ & $E_1= e_1+e_3$, & $E_2= e_2$, & $E_3= te_3$, &$ F_1=f_1$\\
\hline
$(3, 1)_{20} \rightarrow  (3, 1)_{26} $& $E_1=e_1+e_2$, &$E_2=e_3$,& $E_3= te_2$, &$F_1=f_1$\\
\hline
$(3, 1)_{20} \rightarrow  (3, 1)_{30} $& $E_1=e_1$, &$E_2=e_2$,& $E_3= te_3$, &$F_1=f_1$\\
\hline
$(3, 1)_{20} \rightarrow  (3, 1)_{31} $& $E_1=e_3$, &$E_2=e_2$,& $E_3= te_1$, &$F_1=f_1$\\
\hline
$(3, 1)_{21} \rightarrow  (3, 1)_{25} $ & $E_1= e_1+e_3$, & $E_2=e_2$, & $E_3=te_1$, &$ F_1=f_1$\\
\hline
$(3, 1)_{21} \rightarrow  (3, 1)_{27} $ & $E_1= e_1+e_2$, & $E_2=e_3$, & $E_3=te_2$, &$ F_1=f_1$\\
\hline
$(3, 1)_{21} \rightarrow  (3, 1)_{30} $ & $E_1= e_1$, & $E_2=e_2$, & $E_3=te_3$, &$ F_1=f_1$\\
\hline
$(3, 1)_{21} \rightarrow  (3, 1)_{32} $ & $E_1= e_3$, & $E_2=e_2$, & $E_3=te_1$, &$ F_1=f_1$\\
\hline
$(3, 1)_{22} \rightarrow  (3, 1)_{1} $& $E_1=e_1+e_3$, &$E_2=e_2$,& $E_3= te_3$, &$F_1=f_1$\\
\hline
$(3, 1)_{22} \rightarrow  (3, 1)_{24} $ & $E_1= e_1+e_2$, & $E_2=e_3$, & $E_3=te_2$, &$ F_1=f_1$\\
\hline
$(3, 1)_{22} \rightarrow  (3, 1)_{32} $ & $E_1= e_1$, & $E_2=e_3$, & $E_3=te_2$, &$ F_1=f_1$\\
\hline
$(3, 1)_{22} \rightarrow  (3, 1)_{33} $ & $E_1= e_1$, & $E_2=e_2$, & $E_3=te_3$, &$ F_1=f_1$\\
\hline
$(3, 1)_{23} \rightarrow  (3, 1)_{28} $ & $E_1= e_1+e_2$, & $E_2=te_2+e_3$, & $E_3=t^2e_2$, &$ F_1=f_1$\\
\hline
$(3, 1)_{23} \rightarrow  (3, 1)_{34} $& $E_1=e_1$, &$E_2=e_3$,& $E_3= te_2$, &$F_1=f_1$\\
\hline
$(3, 1)_{23} \rightarrow  (3, 1)_{37} $ & $E_1= e_2$, & $E_2=te_1+e_3$, & $E_3=-t^2e_1$, &$ F_1=f_1$\\
\hline
$(3, 1)_{24} \rightarrow  (3, 1)_{2} $& $E_1=e_1+e_2$, &$E_2=2te_1+\frac{1}{2}e_3$,& $E_3= te_3$, &$F_1=f_1$\\
\hline
$(3, 1)_{24} \rightarrow  (3, 1)_{35} $& $E_1=e_1$, &$E_2=e_3$,& $E_3= te_2$, &$F_1=f_1$\\
\hline
$(3, 1)_{24} \rightarrow  (3, 1)_{37} $& $E_1=e_2$, &$E_2=te_1-\frac{1}{2}e_3$,& $E_3= t^2e_1-te_3$, &$F_1=f_1$
\\
\hline
$(3, 1)_{25} \rightarrow  (3, 1)_{3} $& $E_1=e_1+e_2$, &$E_2=2te_1+ \frac{1}{2}e_3$,& $E_3= te_3$, &$F_1=f_1$\\
\hline
$(3, 1)_{25} \rightarrow  (3, 1)_{36} $ & $E_1=e_1$, & $E_2= e_3$, & $E_3= te_2$, &$ F_1=f_1$\\
\hline
$(3, 1)_{25} \rightarrow  (3, 1)_{37} $& $E_1=e_2$, &$E_2=te_1- \frac{1}{2}e_3$,& $E_3= t^2e_1-te_3$, &$F_1=f_1$\\
\hline
$(3, 1)_{26} \rightarrow  (3, 1)_{2} $& $E_1=e_1+e_2$, &$E_2=te_2+e_3$,& $E_3= t^2e_2$, &$F_1=f_1$\\
\hline
$(3, 1)_{26} \rightarrow  (3, 1)_{34} $& $E_1=e_1$, &$E_2=e_3$,& $E_3= te_2$, &$F_1=f_1$\\
\hline
$(3, 1)_{26} \rightarrow  (3, 1)_{38} $& $E_1=e_2$, &$E_2=te_1+e_3$,& $E_3= - t^2e_1$, &$F_1=f_1$\\
\hline
$(3, 1)_{27} \rightarrow  (3, 1)_{3} $& $E_1=e_1+e_2$, &$E_2=te_2+e_3$,& $E_3= t^2e_2$, &$F_1=f_1$\\
\hline
$(3, 1)_{27} \rightarrow  (3, 1)_{34} $& $E_1=e_1$, &$E_2=e_3$,& $E_3= t e_2$, &$F_1=f_1$\\
\hline
$(3, 1)_{27} \rightarrow  (3, 1)_{39} $& $E_1=e_2$, &$E_2=te_1+e_3$,& $E_3
= -  t^2 e_1$, &$F_1=f_1$\\
\hline
$(3, 1)_{28} \rightarrow  (3, 1)_{29} $ & $E_1=e_1$, & $E_2=te_2$, & $E_3=te_3$, &$ F_1=f_1$\\
\hline
$(3, 1)_{28} \rightarrow  (3, 1)_{40}^N $ & $E_1=te_1+e_2$, & $E_2=t^2e_1+2te_2+e_3$, & 
$E_3= t^3e_1+3t^2e_2+3te_3$, &$ F_1=f_1$\\
\hline
$(3, 1)_{29} \rightarrow  (3, 1)_{45}^N $ & $E_1= te_1+te_2$, & $E_2= t^2e_1+2t^2e_2$, & $E_3=te_3$, &$ F_1=f_1$\\
\hline
$(3, 1)_{30} \rightarrow  (3, 1)_{34} $ & $E_1=e_1+ e_2$, & $E_2=t e_2$, & $E_3=t e_3$, &$ F_1=f_1$\\
\hline
$(3, 1)_{30} \rightarrow  (3, 1)_{37} $ & $E_1=e_1$, & $E_2=t e_2+e_3$, & $E_3=t^2 e_2$, &$ F_1=f_1$\\
\hline
$(3, 1)_{30} \rightarrow  (3, 1)_{41} $ & $E_1=e_1$, & $E_2=te_2$, & $E_3=te_3$, &$ F_1=f_1$\\
\hline
$(3, 1)_{31} \rightarrow  (3, 1)_{35} $ & $E_1=e_1+e_2$, & $E_2=t e_2$, & $E_3=e_3$, &$ F_1=f_1$\\
\hline
$(3, 1)_{31} \rightarrow  (3, 1)_{37} $ & $E_1=e_2$, & $E_2=t e_1+e_3$, & $E_3=t^2e_1$, &$ F_1=f_1$\\
\hline
$(3, 1)_{31} \rightarrow  (3, 1)_{38} $& $E_1=e_1$, &$E_2=te_2+e_3$,& $E_3= t^2e_2$, &$F_1=f_1$\\
\hline
$(3, 1)_{32} \rightarrow  (3, 1)_{36} $ & $E_1=e_1+e_2$, & $E_2=t e_2$, & $E_3=e_3$, &$ F_1=f_1$\\
\hline
$(3, 1)_{32} \rightarrow  (3, 1)_{37} $& $E_1=e_2$, &$E_2=te_1+e_3$,& $E_3=  t^2 e_1$, &$F_1=f_1$\\
\hline
$(3, 1)_{32} \rightarrow  (3, 1)_{39} $ & $E_1=e_1$, & $E_2=t e_2+e_3$, & $E_3=t^2e_2$, &$ F_1=f_1$\\
\hline
$(3, 1)_{33} \rightarrow  (3, 1)_{35} $ & $E_1=e_1+e_2$, & $E_2=t e_2$, & $E_3=e_3$, &$ F_1=f_1$\\
\hline
$(3, 1)_{33} \rightarrow  (3, 1)_{39} $& $E_1=e_1$, &$E_2=te_2+e_3$,& $E_3= t^2 e_2$, &$F_1=f_1$\\
\hline
$(3, 1)_{34} \rightarrow  (3, 1)_{40}^N $ & $E_1=te_1+e_2+e_3$, & $E_2=t^2e_1+2te_2$, & $E_3=
t^3e_1+3t^2e_2$, &$ F_1=f_1$\\
\hline
$(3, 1)_{35} \rightarrow  (3, 1)_{40}^N $& $E_1=te_1+e_2+e_3$, &$E_2=t(e_2-e_3)$,& $E_3= t^2 e_2$, &$F_1=f_1$\\
\hline
$(3, 1)_{36} \rightarrow  (3, 1)_{40}^N $& $E_1=te_1+e_2+e_3$, &$E_2=t(e_2-e_3)  $,& $E_3= t^2e_2$, &$F_1=f_1$\\
\hline
$(3, 1)_{37} \rightarrow  (3, 1)_{40}^N $ & $E_1=te_1+e_2$, & $E_2=t^2e_1+e_3$, & $E_3=t^3e_1$, &$ F_1=f_1$\\
\hline
$(3, 1)_{37} \rightarrow  (3, 1)_{41} $ & $E_1=e_1$, & $E_2=te_2$, & $E_3=e_3$, &$ F_1=f_1$\\
\hline
$(3, 1)_{38} \rightarrow  (3, 1)_{40}^N $& $E_1=te_1+e_2$, &$E_2=t^2e_1+e_3$,& $E_3= t^3 e_1$, &$F_1=f_1$\\
\hline
$(3, 1)_{38} \rightarrow  (3, 1)_{42} $ & $E_1=e_1$, & $E_2=te_2$, & $E_3=e_3$, &$ F_1=f_1$\\
\hline
$(3, 1)_{39} \rightarrow  (3, 1)_{40}^N $& $E_1=te_1+e_2$, &$E_2=t^2e_1+e_3$,& $E_3= t^3 e_1$, &$F_1=f_1$\\
\hline
$(3, 1)_{39} \rightarrow  (3, 1)_{43} $ & $E_1=e_1$, & $E_2=te_2$, & $E_3=e_3$, &$ F_1=f_1$\\
\hline
$(3, 1)_{40}^N \rightarrow  (3, 1)_{44}^N $ & $E_1= te_1$, & $E_2=te_2$, & $E_3=t^2e_3$, &$ F_1=f_1$\\
\hline
$(3, 1)_{41} \rightarrow  (3, 1)_{45}^N $ & $E_1=te_1+e_2$, & $E_2=t^2e_1$, & $E_3=e_3$, &$ F_1=f_1$\\
\hline
$(3, 1)_{42} \rightarrow  (3, 1)_{45}^N $& $E_1=te_1+e_2$, &$E_2=t^2e_1$,& $E_3= e_3$, &$F_1=f_1$\\
\hline
$(3, 1)_{43} \rightarrow  (3, 1)_{45}^N $ & $E_1= te_1+e_2$, & $E_2= t^2e_1$, & $E_3= e_3$, &$ F_1=f_1$\\
\hline
$(3, 1)_{44}^N \rightarrow  (3, 1)_{45}^N $ & $E_1= e_1+e_2$, & $E_2=2e_3$, & $E_3=te_1$, &$ F_1=f_1$\\
\hline
$(3, 1)_{46} \rightarrow  (3, 1)_{49} $ & $E_1=e_1$, & $E_2=te_2$, & $E_3=t^2e_3$, &$ F_1=f_1$\\
\hline
$(3, 1)_{46} \rightarrow  (3, 1)_{51} $ & $E_1=e_1$, & $E_2=e_3$, & $E_3=te_2$, &$ F_1=f_1$\\
\hline
$(3, 1)_{47} \rightarrow  (3, 1)_{44}^N $ & $E_1= te_1$, & $E_2=e_2+e_3$, & $E_3= \frac{1}{2}te_2+te_3  $, &$ F_1=f_1$\\
\hline
$(3, 1)_{49} \rightarrow  (3, 1)_{57} $ & $E_1= e_3$, & $E_2= e_1$, & $E_3= te_2$, &$ F_1=f_1$\\
\hline
$(3, 1)_{50} \rightarrow  (3, 1)_{47} $ & $E_1=e_1$, & $E_2=te_2$, & $E_3=e_3$, &$ F_1=f_1$\\
\hline
$(3, 1)_{51} \rightarrow  (3, 1)_{57} $ & $E_1=te_1$, & $E_2=e_2$, & $E_3=e_3$, &$ F_1=f_1$\\
\hline
$(3, 1)_{52} \rightarrow  (3, 1)_{41} $ & $E_1=e_1$, & $E_2=te_2$, & $E_3=e_3$, &$ F_1=f_1$\\
\hline
$(3, 1)_{52} \rightarrow  (3, 1)_{47} $ & $E_1= e_1+e_2$, & $E_2=e_3$, & $E_3=te_1$, &$ F_1=f_1$\\
\hline
$(3, 1)_{53} \rightarrow  (3, 1)_{7} $ & $E_1= e_1+e_2$, & $E_2= e_3$, & $E_3=  t e_2$, &$ F_1=f_1$\\
\hline
$(3, 1)_{53} \rightarrow  (3, 1)_{42} $ & $E_1=e_1$, & $E_2=te_2$, & $E_3=e_3$, &$ F_1=f_1$\\
\hline
$(3, 1)_{54} \rightarrow  (3, 1)_{8} $ & $E_1= e_1+e_2$, & $E_2= e_3$, & $E_3=  t e_2$, &$ F_1=f_1$\\
\hline
$(3, 1)_{54} \rightarrow  (3, 1)_{43} $ & $E_1=e_1$, & $E_2=te_2$, & $E_3=e_3$, &$ F_1=f_1$\\
\hline
$(3, 1)_{55} \rightarrow  (3, 1)_{7} $ & $E_1= e_1+e_2$, & $E_2= e_3$, & $E_3=  t e_2$, &$ F_1=f_1$\\
\hline
$(3, 1)_{55} \rightarrow  (3, 1)_{41} $ & $E_1= e_1$, & $E_2= te_2$, & $E_3= e_3$, &$ F_1=f_1$\\
\hline
$(3, 1)_{56} \rightarrow  (3, 1)_{8} $ & $E_1= e_1+e_2$, & $E_2= e_3$, & $E_3=  t e_2$, &$ F_1=f_1$\\
\hline
$(3, 1)_{56} \rightarrow  (3, 1)_{41} $ & $E_1= e_1$, & $E_2= te_2$, & $E_3= e_3$, &$ F_1=f_1$\\
\hline
$(3, 1)_{57} \rightarrow  (3, 1)_{44}^N $ & $E_1=e_1+e_3$, & $E_2=te_2$, & $E_3=\frac{1}{2}te_3 $, &$ F_1=f_1$\\
\hline
$(3, 1)_{58} \rightarrow  (3, 1)_{44}^N $ & $E_1= e_1+e_3$, & $E_2= te_2$, & $E_3= \frac{1}{2} t e_3$, &$ F_1=f_1$\\
\hline
$(3, 1)_{59} \rightarrow  (3, 1)_{44}^N $ & $E_1= e_1+e_3$, & $E_2= t e_2$, & $E_3= \frac{1}{2} t e_3$, &$ F_1=f_1$\\
\hline
\end{longtable}

In order to find the irreducible components of the variety $\mathcal{JS}^{(3,1)}$, we looking for the rigid superalgebras in $\mathcal{JS}^{(3,1)}$. Notice that since there exist a finite number of isomorphism classes in $\mathcal{JS}^{(3,1)}$, it follows that a Jordan superalgebra $\mathcal{J}$ is rigid if and only if 
every deformation of $\mathcal{J}$ is isomorphic to $\mathcal{J}$, (see \cite{tesejenny}, Proposition 4.39, p. 70).

\begin{lemma}Let $\mathcal{J}=\mathcal{J}_0\oplus \mathcal{J}_1 \in \mathcal{JS}^{(3,1)}$. If the Jordan algebra $\mathcal{J}_0$ is rigid;  so also is $\mathcal{J}$.
\label{lemma:rigid_even_part}
\end{lemma}
\begin{proof} 

Comparison of Table \ref{table:3_1} %\ref{table:3_1_continuation},
and the results given in  
\cite{irynashesta},
%\cite{tesejenny} 
 it shows that
$\dim(\Aut(\mathcal{J})) = \dim(\Aut(\mathcal{J}_0))+1$, for all $\mathcal{J} \in \mathcal{JS}^{(3,1)}$.  
 We now suppose that there exists $ \mathcal{J}^\prime \in \mathcal{JS}^{(3,1)}$ such that  $ \mathcal{J}^\prime \not\simeq \mathcal{J}$ and   $\mathcal{J}^\prime  \to \mathcal{J}$. From Lemma \ref{lemma:invariants}.(iii), it follows that $\mathcal{J}^\prime_0 \to \mathcal{J}_0$. Hence, since $\mathcal{J}_0$ is rigid,  it may be concluded that $\mathcal{J}^\prime_0 \simeq \mathcal{J}_0$.  Consequently,  
$\dim( \Aut( \mathcal{J}^\prime)) = \dim(\Aut(\mathcal{J}))$, and  by Lemma \ref{lemma:invariants}.(i)  we get a contradiction. 
\end{proof}

 The principal result of this section is the following. 
 \begin{theorem} There exist exactly $21$ rigid Jordan superalgebras of type $(3,1)$, whose closures of their orbits are the irreducible components of the variety  $\mathcal{JS}^{(3,1)}$. The rigid superalgebras are:\\ $(3,1)_i$, for 
$i \in \{6, 9, 10, 11, 13, 14, 15, 16, 18, 19, 20, 21, 22,  46, 48, 50, 52, 53, 54, 55, 56 \}.
$
\label{theorem:irreducible_components_3_1}
\end{theorem}

\begin{proof}
First, we deal with the existence of the $21$ rigid superalgebras in $\mathcal{JS}^{(3,1)}$. 
 For $i \in \{6, 9, 10, 13, 14, 18, 19, 20, 21, 22,  46, 48, 50, 52, 53, 54, 55, 56 \}$, we have that 
$((3,1)_i)_0$ is a rigid Jordan algebra (see  \cite{irynashesta}).  Hence, from  Lemma  \ref{lemma:rigid_even_part}  we conclude that  the superalgebra 
$(3,1)_i$ is rigid.  
 We will now show that  $(3,1)_{11}$, $(3,1)_{15}$ and $(3,1)_{16}$ are rigid by another method. 
For simplicity, we argue  only for $(3,1)_{11}$,  since the other two cases are analogous. 
Since  $\dim (\Aut (3,1)_{11})=3$, from Lemma \ref{lemma:invariants}.(i),  it is enough to show that if $\mathcal{J}$ is a superalgebra  such that $\dim (\Aut (\mathcal{J}))\leq 2$,  then   $\mathcal{J} \not  \to  (3,1)_{11}$. 
 First notice that $((3,1)_{11})_0$ is not rigid because $((3,1)_6)_0 =((3,1)_{46})_0 \to ((3,1)_{11})_0$, although $(3,1)_6 \not \to (3,1)_{11}$ and  $(3,1)_{46} \not \to (3,1)_{11}$, because they have structure of Jordan algebras and  $(3,1)_6   \not\to (3,1)_{11}$, $(3,1)_{46}  \not\to (3,1)_{11}$ as algebras (see Remark \ref{Remark:como algebras}). %Thereby, by Lemma {\bf X?}, $(3,1)_6 \not \to (3,1)_{11}$ and $(3,1)_{46} \not \to (3,1)_{11}$. 
Now, if 
$
\mathcal{J} \in \{\mathcal{J} \in \mathcal{JS}^{(3,1)} | \dim(\Aut(\mathcal{J})) \leq 2   
\} \setminus \{ (3,1)_6, (3,1)_{46}   \} 
$, 
then 
$
\mathcal{J}_0 \not \to ((3,1)_{11})_0 
$ (see \cite{tesejenny}, Table 5.2, p. 95). Hence, $\mathcal{J} \not \to (3,1)_{11}$ and  this shows that $(3,1)_{11}$ is rigid.  Finally,  by
Table
\ref{table:3_1_degenerations_first}, %\ref{table:3_1_degenerations_second} and \ref{table:3_1_degenerations_third},
%shows that 
there is no other rigid superalgebra in $\mathcal{JS}^{(3,1)}$.
\end{proof}

\begin{remark}
With the previous information we were able to determine in $99.05 \%$ of cases whether a Jordan superalgebra of type $(3,1)$ belongs or not to the Zariski closure of the orbit of another superalgebra of the same type. This allowed us to describe the irreducible components of the variety $\mathcal{JS}^{(3,1)}$ as being the closure of the orbits of rigid Jordan superalgebras:

\begin{align*}
\overline{(3,1)_6^G}   =&   \{{\color{red} (3,1)_2}, {\color{red}  (3,1)_3}, (3,1)_4,  (3,1)_ 6, (3,1)_{12}, (3,1)_{17}, {\color{red}(3,1)_{28}},  {\color{red}  (3,1)_{40}^N},  (3,1)_{44}^N,\\
&(3,1)_{45}^N, (3,1)_{59}, (3,1)_{60}^N \} . \\
%%%
\overline{(3,1)_9^G} =&  \{ (3,1)_9,  (3,1)_{60}^N\}. \\
%%%
\overline{(3,1)_{10}^G } =& \{(3,1)_{10}, (3,1)_{60}^N \} .\\
%%%
\overline{(3,1)_{11}^G }   =& \{  (3,1)_{11},  {\color{red}(3,1)_{40}^N},  (3,1)_{44}^N, (3,1)_{45}^N,   (3,1)_{58},  (3,1)_{60}^N\}.  \\
%%%
\overline{(3,1)_{13}^G}   =&  \{ (3,1)_{7}, (3,1)_{13}, {\color{red}(3,1)_{40}^N},  (3,1)_{44}^N, (3,1)_{45}^N, (3,1)_{60}^N\} . \\
%%%
\overline{(3,1)_{14}^G}  =& \{ (3,1)_8, (3,1)_{14}, {\color{red}(3,1)_{40}^N},  (3,1)_{44}^N, (3,1)_{45}^N, (3,1)_{60}^N \}.  \\
%%%
\end{align*}
\begin{align*}
\overline{(3,1)_{15}^G }  =& \{   (3,1)_5,  (3,1)_{15}, {\color{red} (3,1)_{40}^N}, (3,1)_{44}^N, (3,1)_{45}^N,  (3,1)_{57}, (3,1)_{59},   (3,1)_{60}^N\} .\\
%%%
\overline{(3,1)_{16}^G }  =&  \{  (3,1)_4,  (3,1)_{16},  {\color{red}(3,1)_{40}^N}, (3,1)_{44}^N, (3,1)_{45}^N, (3,1)_{57}, (3,1)_{58},  (3,1)_{60}^N \} . \\
%%%
\overline{(3,1)_{18}^G }  =&   \{ (3,1)_7, (3,1)_{18}, {\color{red} (3,1)_{40}^N}, (3,1)_{43},  (3,1)_{44}^N,  (3,1)_{45}^N, (3,1)_{60}^N\} . \\
%%%
\overline{ (3,1)_{19}^G }  =&   \{ (3,1)_{19}, (3,1)_{23},  (3,1)_{28},  (3,1)_{29}, (3,1)_{30},   (3,1)_{34},  (3,1)_{37},  (3,1)_{40}^N,  (3,1)_{41},\\
&(3,1)_{44}^N, (3,1)_{45}^N, (3,1)_{60}^N\}.  \\
%%%
\overline{(3,1)_{20}^G} =& \{ (3,1)_2,  (3,1)_4, (3,1)_{20},  
 {\color{red} (3,1)_{23}}, 
 (3,1)_{24}, (3,1)_{26},   {\color{red} (3,1)_{28}},     {\color{red}   (3,1)_{29}}, (3,1)_{30},\\
 &(3,1)_{31},  (3,1)_{34}, (3,1)_{35},  (3,1)_{37}, (3,1)_{38}, (3,1)_{40}^N,   (3,1)_{41},   (3,1)_{42},  (3,1)_{44}^N, \\   
 &(3,1)_{45}^N, (3,1)_{60}^N  \}  .  \\
%%%
\overline{(3,1)_{21}^G }  =& \{ (3,1)_3, (3,1)_5, (3,1)_{21}, (3,1)_{25}, (3,1)_{27},  (3,1)_{30}, (3,1)_{32}, (3,1)_{34}, (3,1)_{36},\\
&(3,1)_{37}, (3,1)_{39}, (3,1)_{40}^N, (3,1)_{41},  (3,1)_{43}, (3,1)_{44}^N,    (3,1)_{45}^N, (3,1)_{60}^N \} . \\
%%%
\overline{(3,1)_{22}^G}  =&  \{ (3,1)_1,(3,1)_2, (3,1)_4,  (3,1)_{22}, (3,1)_{24}, (3,1)_{32}, (3,1)_{33},  (3,1)_{35}, (3,1)_{36},\\
&(3,1)_{37}, (3,1)_{39}, (3,1)_{40}^N, (3,1)_{44}^N,  {\color{red} (3,1)_{28}} , {\color{red}  (3,1)_{29}}, (3,1)_{41}, (3,1)_{43},  (3,1)_{45}^N,\\
&(3,1)_{60}^N \} .\\
%%%
\overline{(3,1)_{46}^G } =&   \{ {\color{red}(3,1)_{28}}, (3,1)_{29},   {\color{red} (3,1)_{40}^N},  (3,1)_{44}^N,   (3,1)_{45}^N, (3,1)_{46},    (3,1)_{49}, (3,1)_{51},  
   (3,1)_{57}, \\
   &(3,1)_{60}^N\} . \\
%%%%
(\overline{3,1)_{48} ^G} =&\{ (3,1)_{48}, (3,1)_{60}^N\} .\\
%%%
\overline{(3,1)_{50}^G } =&  \{ {\color{red}(3,1)_{40}^N}, (3,1)_{44}^N, (3,1)_{45}^N,  (3,1)_{47}, (3,1)_{50},  (3,1)_{60}^N\}. \\
%%%
\overline{(3,1)_{52}^G}  =& \{  {\color{red}(3,1)_{40}^N},  (3,1)_{41},  (3,1)_{44}^N,  (3,1)_{45}^N,   (3,1)_{47},  (3,1)_{52}, (3,1)_{60}^N \} .\\
%%%
\overline{(3,1)_{53}^G}  =& \{  (3,1)_7,  {\color{red}(3,1)_{40}^N}, (3,1)_{42}, {\color{red}(3,1)_{43}},  (3,1)_{44}^N,  (3,1)_{45}^N,  (3,1)_{53},   (3,1)_{60}^N \} .\\
%%%
\overline{(3,1)_{54}^G } =&  \{ (3,1)_7,  {\color{red}(3,1)_{40}^N},  (3,1)_{43}, (3,1)_{44}^N,   (3,1)_{45}^N,  (3,1)_{ 54}, (3,1)_{60}^N   \}. \\
%%%
\overline{(3,1)_{55} ^G} =& \{  (3,1)_7,  {\color{red}(3,1)_{40}^N},  {\color{red}(3,1)_{47}}, (3,1)_{41},   (3,1)_{44}^N,  (3,1)_{45}^N,   (3,1)_{55}, (3,1)_{60}^N\} .\\
%%%
\overline{(3,1)_{56}^G  } =& \{ (3,1)_8,  {\color{red}(3,1)_{40}^N},   (3,1)_{41},   (3,1)_{44}^N, (3,1)_{45}^N,  (3,1)_{56}, (3,1)_{60}^N\} .
\end{align*}

Although we were unable to find some deformations (indicated in red above) this does not affect the number of irreducible components of the variety. Furthermore, in Table \ref{table:open_problems_3_1} we have used the symbol $\stackrel{?}{\rightarrow}$ to represent that such information is unknown.

\begin{longtable}[H]{|c | c|}
\caption{ \label{table:open_problems_3_1}Open problems in $\mathcal{JS}^{(3,1)}$}
\\
\hline
 $\mathcal{J}  \stackrel{?}{\rightarrow}  \mathcal{J}^\prime$  & Reason\\
 \hline
 \endfirsthead
 $ 6 \stackrel{?}{\rightarrow} i$, for  $ i\in \{2, 3, 28, 40  \}$;  $ \;\; 46 \stackrel{?}{\rightarrow} 28$;  $ \;\;46 \stackrel{?}{\rightarrow} 40$;  &   $\mathcal{J}_0  \stackrel{?}{\rightarrow}  \mathcal{J}_0^\prime$   \\
  $ i  \stackrel{?}{\rightarrow} 40$, for  $ i \in \{  11, \dots, 18, 49 , \dots, 56 \}$.
& \\
 \hline
 $ 20 \stackrel{?}{\rightarrow} \{23, 28, 29\}$; 
$ \;\;21 \stackrel{?}{\rightarrow} 42$;    $ \;\;22 \stackrel{?}{\rightarrow} 29$;  $ \;\; 24 \stackrel{?}{\rightarrow} 29$;     & $ \mathcal{J}  \rightarrow  \mathcal{J}^\prime$ as algebras\\
 $ 27 \stackrel{?}{\rightarrow} 42$;   $\;\; 31 \stackrel{?}{\rightarrow} 34$;  $\;\; 55 \stackrel{?}{\rightarrow} 47 $;   $ \;\;53 \stackrel{?}{\rightarrow} 43 $.  &   but $ \mathcal{J}  \stackrel{?}{\rightarrow}  \mathcal{J}^\prime$ as superalgebras   \\
 \hline
  $ 21 \stackrel{?}{\rightarrow} 38$;   $ \;\;22 \stackrel{?}{\rightarrow} 28$.  &   $\mathcal{J}  \stackrel{?}{\rightarrow}  \mathcal{J}^\prime$ as algebras \\
\hline
\end{longtable}
\end{remark}

The irreducible components are indicated in the Hasse diagram in Figure \ref{grafica_3_1} following the same pattern of colors and formats as in the previous section.

\begin{figure}[H]
  \centering
    \includegraphics[width=18.1cm, height=11.5cm, angle=90]{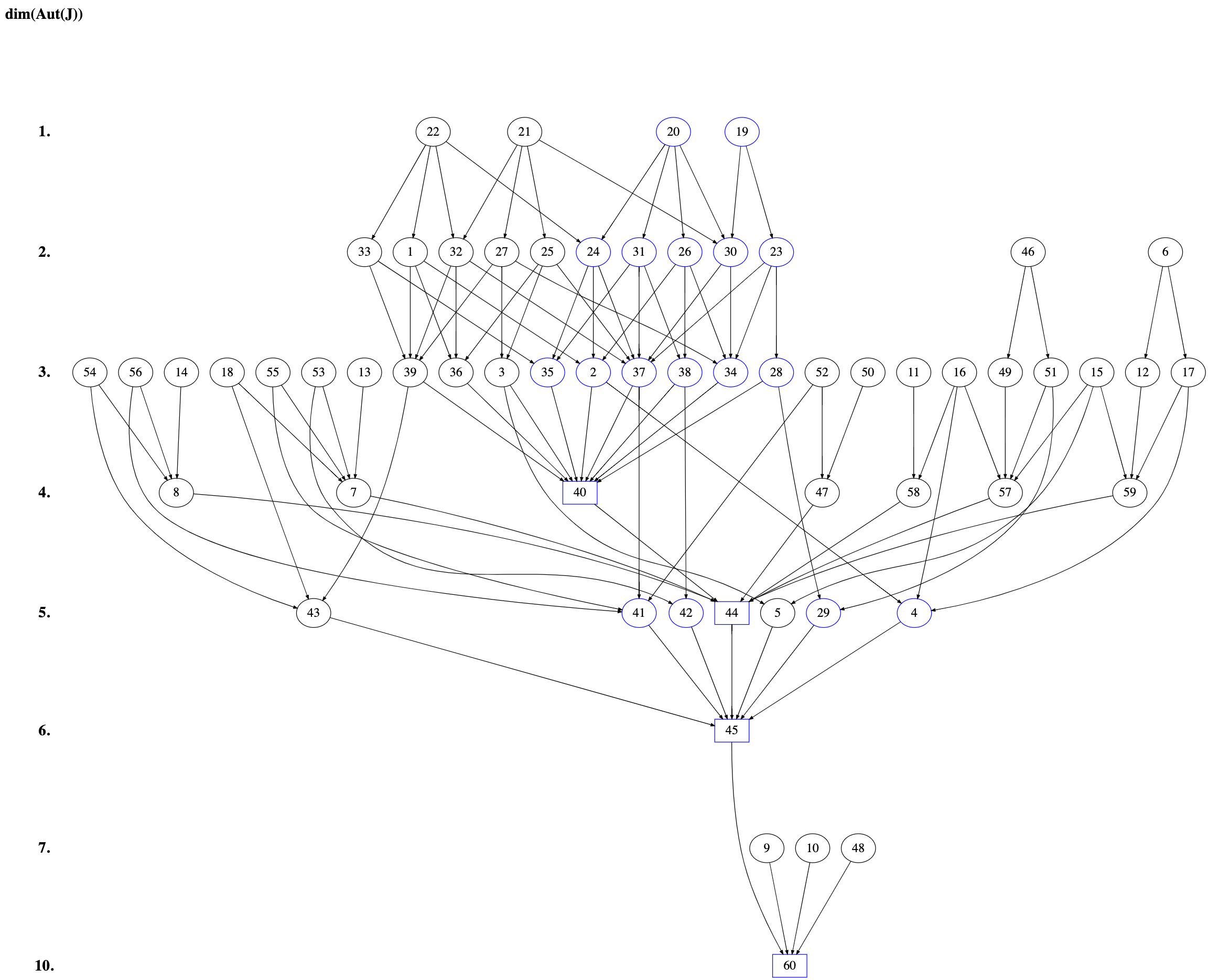}
  \caption{\label{grafica_3_1}
  Hasse diagram of deformations for Jordan superalgebras of type $(3,1)$
  }
\end{figure}

Finally, as a direct consequence of Theorem \ref{theorem:irreducible_components_3_1}, we obtain the following results:

\begin{corollary}
The subvariety $\mathcal{ASC}^{(3,1)} \subset \mathcal{JS}^{(3,1)}$ of supercommutative associative superalgebras of type $(3,1)$ has $2$ irreducible components given by 
\begin{align*}
\overline{ (3,1)_{19}^G }   = & \{   
(3,1)_{19},  (3,1)_{23}, (3,1)_{28}, (3,1)_{29},
(3,1)_{30},  (3,1)_{34},
(3,1)_{37}, 
(3,1)_{40}^N, (3,1)_{41}, \\ 
& (3,1)_{42}, 
(3,1)_{44}^N, 
(3,1)_{45}^N, (3,1)_{60}^N
\}.\\
%%%
\overline{ (3,1)_{20}^G }  = & \{  
(3,1)_{2}, (3,1)_{4},
(3,1)_{20}, (3,1)_{24}, (3,1)_{26},  
(3,1)_{30}, (3,1)_{31}, (3,1)_{34}, (3,1)_{35}, \\
& (3,1)_{37}, (3,1)_{38},  (3,1)_{40}^N, (3,1)_{41}, (3,1)_{42}, (3,1)_{44}^N,  (3,1)_{45}^N, (3,1)_{60}^N
   \} .\\
%%%
\end{align*}
\end{corollary}

\begin{corollary}
The subvariety $\mathcal{NJS}^{(3,1)} \subset \mathcal{JS}^{(3,1)}$ of nilpotent Jordan superalgebras of type $(3,1)$ is irreducible and  it is  given by 
$$
\overline{ ((3, 1)^N_{40})^G } =  \{ 
(3, 1)^N_{40}, (3, 1)^N_{44}, (3, 1)^N_{45}, (3, 1)^N_{60}
\}
$$
\label{corollary:irreducible_components_nilpotent_3_1}
\end{corollary}

\bibliographystyle{amsplain}
\bibliography{library}
\end{document}